\documentclass[oneside]{amsart}
\usepackage{amsmath,amsthm,latexsym,amssymb,hyperref,amsfonts}
\usepackage{stmaryrd}
\usepackage{eucal}
\usepackage{mathrsfs}
\usepackage{comment}
\usepackage[all]{xy}
\usepackage[usenames,dvipsnames]{xcolor}
\usepackage{pdflscape}
\usepackage{tikz}
\allowdisplaybreaks
\usetikzlibrary{matrix,arrows}
\usepackage{todonotes}
\usepackage{enumitem}
\usepackage[textwidth=13.5cm, textheight=21cm]{geometry}

\numberwithin{equation}{section}

\newtheorem{thm}{Theorem}[section]
\newtheorem{lem}[thm]{Lemma}
\newtheorem{prop}[thm]{Proposition}
\newtheorem{cor}[thm]{Corollary}

\newtheorem*{thrm}{Main Theorem}
\theoremstyle{definition}
\newtheorem{defn}[thm]{Definition}

\theoremstyle{remark}
\newtheorem{rmk}[thm]{Remark}

\DeclareMathAlphabet{\mathbbe}{U}{bbold}{m}{n}

\def\endash{\mathchar"707B}

\newdir{ >}{{}*!/-10pt/@{>}}

\newcommand{\Id}{\operatorname{Id}}

\renewcommand\hom{\operatorname{Hom}}
\DeclareMathOperator{\colim}{colim}

\newcommand{\lra}{\longrightarrow}

\newcommand{\Ch}{\mathsf{Ch}}

\newcommand{\A}{\mathsf A}

\newcommand{\cA}{\mathcal A}
\newcommand{\cE}{\mathcal E}
\newcommand{\cF}{\mathcal F}
\newcommand{\cO}{\mathcal O}
\newcommand{\oF}{\overline {\mathcal F}}
\newcommand{\cD}{\mathcal D}
\newcommand{\cG}{\mathcal G}
\newcommand{\cT}{\mathcal T}

\newcommand{\cK}{\mathcal K}

\newcommand{\cC}{\mathcal C}

\newcommand{\bQ}{\mathbb Q}

\newcommand{\bR}{\mathbb R}
\newcommand{\mA}{\mathrm A}

\newcommand{\tcD}{\widetilde{\mathcal D}}
\newcommand{\tcT}{\widetilde{\mathcal T}}

\newcommand\cL{\mathcal L}

\newcommand{\cP}{\mathcal{P}}

\newcommand\cell {\endash \mathrm {cell}\endash}

\newcommand\adjunct[4]{\xymatrix@C=2pc@R=2pc{#1\ar @<1.25ex>[rr]^{#3}&\perp&#2\ar @<1.25ex>[ll]^{#4}}}

\newcommand{\RR}{\mathbb{R}}

\begin{document}
\title{An algebraic model for rational $SO(3)$--spectra}

\author[K\c{e}dziorek]{Magdalena K\c{e}dziorek}

\address{MATHGEOM \\
    \'Ecole Polytechnique F\'ed\'erale de Lausanne \\
    CH-1015 Lausanne \\
    Switzerland}
\email{magdalena.kedziorek@epfl.ch}

\date {\today }


 \begin{abstract}
Greenlees established an equivalence of categories between the homotopy category of rational $SO(3)$--spectra and the derived category $D\cA(SO(3))$ of a certain abelian category. In this paper we lift this equivalence of homotopy categories to the level of Quillen equivalences of model categories. Methods used in this paper provide the first step towards obtaining an algebraic model for the toral part of rational $G$--spectra, for any compact Lie group $G$.
 \end{abstract}
\maketitle
\tableofcontents
\section{Introduction}

\paragraph{\bf{Modelling the category of rational $G$--spectra}}
This paper is a contribution to the study of $G$--equivariant cohomology theories and gives a complete analysis for one class of theories, namely rational $SO(3)$--equivariant cohomology theories. To start with, let $G$ be a compact Lie group. Recall that $G$--equivariant cohomology theories are represented by $G$--spectra, so that the category of $G$--equivariant cohomology theories is equivalent to the homotopy category of $G$--spectra. The category of $G$--spectra is quite complicated, with rich structures coming from two sides: topological and group action, and one cannot expect a complete analysis of either cohomology theories or spectra integrally.

Rationalising this category reduces topological complexity simplifying it greatly. At the same time interesting equivariant behaviour remains. In order to understand this behaviour, we try to find a purely algebraic description of the category, that is an algebraic model category Quillen equivalent to the category of rational $G$--spectra. As a result, the homotopy category of the algebraic model is equivalent to the rational stable $G$--homotopy category via triangulated equivalences. Moreover all the homotopy information, such as homotopy limits, in both is the same. 

The conjecture by Greenlees states that for any compact Lie group $G$ there is a nice graded abelian category $\cA(G)$, such that the category $d\cA(G)$ of differential objects in $\cA(G)$ with a certain model structure is Quillen equivalent to the category of rational $G$--spectra
$$G\endash\text{Sp}_\bQ \simeq_Q d\cA(G).$$ 
If we find such $d\cA(G)$ we say that $d\cA(G)$ is an \emph{algebraic model} for rational $G$--spectra. \\

\paragraph{\textbf{Existing work}}
There are several examples of specific Lie groups $G$ for which an algebraic model has been given. Firstly, when $G$ is trivial, it was shown in \cite[Theorem 1.1]{ShipleyHZ} that rational spectra are monoidally Quillen equivalent to chain complexes of $\bQ$--modules. An algebraic model for rational $G$--spectra for finite $G$ is described in \cite[Example 5.1.2]{SchwedeShipleyMorita} and simplified in \cite{BarnesFiniteG} and \cite{KedziorekExceptional}.   An algebraic model for rational torus equivariant spectra was presented in \cite{GreenShipleyT-equiv}, whereas a slightly different approach in \cite{BGKS} gives a \emph{symmetric monoidal} algebraic model for $SO(2)$. This was recently used in \cite{BarnesO(2)new} to provide an algebraic model for rational $O(2)$--spectra.

However, there is no algebraic model known for the whole category of rational $G$--spectra for an arbitrary compact Lie group $G$.  A first step in this direction, a model for rational $G$--spectra over an exceptional subgroup (see Definition \ref{defn:exc_subgp}) for any compact Lie group $G$, was provided in \cite{KedziorekExceptional}. This result is used in Section \ref{Exceptional part_chapter}. \\

\paragraph{\textbf{The group $SO(3)$}}
The group $SO(3)$ is the group of rotations of $\bR^3$. This is the natural next candidate to analyse on the way to understand the behaviour of $d\cA(G)$ for an arbitrary compact Lie group $G$. Notice that $SO(3)$ is significantly more complicated than all groups considered so far, since  it is the first group where the maximal torus is not normal in the whole group. Dealing with this complication shows a method to provide an algebraic model for a part of rational $G$--spectra called \emph{toral} for any compact Lie group $G$. The toral part models those $G$--spectra whose geometric isotropy is a set of subgroups of the maximal torus and corresponds to cohomology theories with toral support. We discuss this further in Remark \ref{rmk:generalGtoN}. \\

\paragraph{\textbf{Main result}}
Let $G$ be $SO(3)$. In this paper we work with orthogonal $G$--spectra, see \cite{MandellMay}. By \cite{Barnes_Splitting}, the category of rational $SO(3)$--orthogonal spectra splits into three parts: toral, dihedral and exceptional. This uses idempotents of the rational Burnside ring $\A(SO(3))_\bQ$ (see Section \ref{ideBurnside}), and reflects a similar splitting at the level of homotopy category. 

The toral part models rational $SO(3)$--spectra with geometric isotropy in the family of subconjugates of the maximal torus $SO(2)$ in $SO(3)$. The dihedral part models rational $SO(3)$--spectra with geometric isotropy in the collection of subgroups $\cD$, which consists of all dihedral subgroups of order greater than $4$ and $O(2)$. The last part, which we call the exceptional part, models rational $SO(3)$--spectra with geometric isotropy in the collection of subgroups $\cE$, which consists of all remaining subgroups (see Section \ref{subgroups}). Thus we are able to work with each of these three parts separately to obtain an algebraic model for rational $SO(3)$--spectra.

The main result of this paper is as follows.
\begin{thrm}
There is a zig--zag of Quillen equivalences between rational $SO(3)$--orthogonal spectra and the algebraic category $d\cA(SO(3))$. 
\end{thrm}
The category $d\cA(SO(3))$, which we call the \textbf{algebraic model for rational $SO(3)$--spectra}, is a product of three parts, which reflects the splitting of the category of rational $SO(3)$--spectra 

$$d\cA(SO(3))\cong d\cA(SO(3),\cT) \times \Ch(\cA(SO(3),\cD))\times \prod_{(H), H\in \cE}\Ch(\bQ[W_{SO(3)}H]).$$

Here $d\cA(SO(3),\cT)$ is the \textbf{algebraic model for the toral part} described in Section \ref{The category A(SO(3),c)}, $\Ch(\cA(SO(3),\cD))$ is the \textbf{algebraic model for the dihedral part} described in Section \ref{A(D)} and $\Ch(\bQ[W_{SO(3)}H])$ is the \textbf{algebraic model for the rational $SO(3)$--spectra over an exceptional subgroup $H$} discussed in Section \ref{chainExcep}.
Since $\cA(SO(3),\cT)$ is a graded abelian category we use the notation $d\cA(SO(3),\cT)$ for differential objects in there. We use notation $\Ch(\cA(SO(3),\cD))$ for differential graded objects (i.e. chain complexes) in $\cA(SO(3),\cD)$, since $\cA(SO(3),\cD)$ doesn't have a grading.

The Main Theorem follows from Proposition \ref{prop:splittingSO(3)}, Theorem \ref{thm:toral_result}, Theorem \ref{thm:summary_d} and Theorem \ref{thm:exceptional_part}.\\

\paragraph{\textbf{Contribution of this paper}}The new idea in this paper concerns the toral part in Section \ref{cyclicPart}. Since the maximal torus is not normal in $SO(3)$ the algebraic model for the toral part gets more complicated than that for the torus \cite{GreenleesS1}, \cite{BGKS} or $O(2)$ \cite{BarnesO(2)new}. To control these complications we use the following method. We consider the restriction--coinduction adjunction between the toral part of rational $SO(3)$--spectra and the toral part of rational $O(2)$--spectra. Here $O(2)$ is the normaliser of the maximal torus in $SO(3)$. This adjunction is a Quillen adjunction, but not a Quillen equivalence. 

However, the cellularisation principle of \cite{GreenShipCell} (see Section \ref{sec:cellularisation} for the definition of cellularisation) gives a Quillen equivalence between the toral part of rational $SO(3)$--spectra and a certain cellularisation of the toral part of rational $O(2)$--spectra see Theorem \ref{CyclicRestrictionEquivalence}. Now it is enough to cellularise the algebraic model for the toral part of rational $O(2)$--spectra and simplify this category, see Section \ref{section_toral_reduction_to_O2}, to obtain the model for the toral part of rational $SO(3)$--spectra.

The idea of using the restriction--coinduction adjunction between the toral part of rational $G$--spectra and the toral part of rational $N_G\mathbb{T}$--spectra (where $\mathbb{T}$ is the maximal torus in $G$) together with the cellularisation principle allows one to provide an algebraic model for the toral part of rational $G$--spectra, for any compact Lie group $G$ \cite{BGK}.

The method to obtain the algebraic model for the dihedral part of rational $SO(3)$--spectra is a slight alteration of the method for the dihedral part for rational $O(2)$--spectra from \cite{BarnesO(2)new} and is presented in Section \ref{subsec:comparison_dihedral}. Some changes in the proof from \cite{BarnesO(2)new} are needed to take into account the fact that our dihedral part excludes subgroups conjugate to $D_2$ and $D_4$ (for reasons explained in Section \ref{subgroups}), whereas the dihedral part of $O(2)$--spectra contains them. However, the idea of the proof remains the same.

Finally, an algebraic model of the exceptional part is an application of the methods from \cite{KedziorekExceptional}. We point out that this is the only part of the paper that considers monoidal structures and gives a monoidal algebraic model.\\

\paragraph{\textbf{Outline of the paper}}
This paper is structured as follows. In Section \ref{General results for SO(3)_chapter} we present some general results about subgroups of $SO(3)$, its rational Burnside ring $\A(SO(3))_\bQ$ and the idempotents used to split the category of rational $SO(3)$--spectra into three parts: toral, dihedral and exceptional (Proposition \ref{prop:splittingSO(3)}). Section \ref{cyclicPart} is the heart of this paper. It contains the description of the algebraic model for the toral part of rational $SO(3)$--spectra. It also presents Quillen equivalences used in obtaining this algebraic model from the algebraic model for toral rational $O(2)$--spectra. Section \ref{Dihedral part_chapter} contains the algebraic model for the dihedral part. Finally, in Section \ref{Exceptional part_chapter} we recall the results from \cite{KedziorekExceptional} to give an algebraic model for the exceptional part of rational $SO(3)$--spectra.\\

\paragraph{\textbf{Notation}}We will stick to the convention of drawing the left adjoint above (or to the left of) the right one in any adjoint pair. We use the notation $G\endash\mathrm{Sp}$ for the category of $G$--equivariant orthogonal spectra. \\

\paragraph{\textbf{Acknowledgments}}
This work is based on a part of my PhD thesis (under the supervision of John Greenlees) and I would like to thank John Greenlees and David Barnes for many useful discussions and comments.

\section{General results for SO(3)}\label{General results for SO(3)_chapter}

We start this part by considering the closed subgroups of $SO(3)$ in Section \ref{subgroups}. We discuss the space $\cF(G)/G$, which is the orbit space of all closed subgroups with finite index in their normaliser, where the topology is induced from the Hausdorff metric, see \cite[Section V.2]{LewisMaySteinberger}.In Section \ref{sec:BousfieldAndCell} we recall two ways of changing a given stable model structure: left Bousfield localisation at an object and cellularisation. We will use these techniques repeatedly throughout the paper. In Section \ref{ideBurnside} we discuss the idempotents of the rational Burnside ring $\A(SO(3))_\bQ$ and the induced splitting of rational $SO(3)$--orthogonal spectra. The main part of Section \ref{ideBurnside} consists of the analysis of two adjunctions: the induction--restriction and restriction--coinduction adjunctions in relation to localisations of categories of equivariant spectra at idempotents.

\subsection{Closed subgroups of $SO(3)$}\label{subgroups}

Recall that $SO(3)$ is a group of rotations of $\RR^3$. We choose a maximal torus $T$ in $SO(3)$ with rotation axis the $z$-axis. We divide the closed subgroups of $G$ into three types: {\bf toral $\cT$, dihedral $\cD$ and exceptional $\cE$}. This division is motivated by the choice of idempotents in the rational Burnside ring for $SO(3)$, that we will use to split the category of rational $SO(3)$--spectra. 

The toral part consist of all tori in $SO(3)$ and all cyclic subgroups of these tori. Note that for any natural number $n$ there is one conjugacy class of subgroups from the toral part of order $n$ in $SO(3)$. 

The dihedral part consists of all dihedral subgroups $D_{2n}$ (dihedral subgroups of order $2n$) of $SO(3)$ where $n$ is greater than 2, together with all subgroups $O(2)$. Note that $O(2)$ is the normaliser for itself in $SO(3)$. Moreover there is only one conjugacy class of a dihedral subgroup $D_{2n}$ for each $n$ greater than $2$, and the normaliser of $D_{2n}$ in $SO(3)$ is $D_{4n}$ for $n>2$. Notice that we excluded subgroups in the conjugacy classes of $D_2$ and $D_4$ from this part. Conjugates of $D_2$ are excluded from the dihedral part, since $D_2$ is conjugate to $C_2$ in $SO(3)$ and that subgroup is already taken into account in the toral part. Conjugates of $D_4$ are excluded from the dihedral part since its normalizer in $SO(3)$ is $\Sigma_4$ (symmetries of a cube), thus its Weyl group $\Sigma_4/D_4$ is of order $6$, whereas all other finite dihedral subgroups $D_{2n}, n>2$ have Weyl groups of order $2$. For simplicity we decided to treat $D_4$ seperately, and put it into the exceptional part. 

There are five conjugacy classes of subgroups which we call exceptional, namely $SO(3)$ itself, the rotation group of a cube $\Sigma_4$,  the rotation group of a tetrahedron $A_4$, the rotation group of a dodecahedron $A_5$  and $D_4$, the dihedral group of order 4. Normalisers of these exceptional subgroups are as follows: $\Sigma_4$ is equal to its normaliser, $A_5$ is equal to its normaliser and the normaliser of $A_4$ is $\Sigma_4$, as is the normaliser of $D_4$.

Consider the space $\cF(SO(3))/SO(3)$ of conjugacy classes of subgroups of $SO(3)$ with finite index in their normalisers. The topology on this space is induced by the Hausdorff metric. We will use this space for choosing idempotents of the rational Burnside ring in Section \ref{ideBurnside}. 

\begin{center}
\begin{tikzpicture}[dot/.style={circle, inner sep=1pt,fill,label={#1},name=#1,color=black}]
\usetikzlibrary{calc}
\coordinate[label=$\mathrm{Space}\ \cF(SO(3))/SO(3)$] (O) at (3,8);
\coordinate[label=$\mathrm{Part\ (Subspace)}$] (O) at (-3,8);
\coordinate[label=$\cE$] (F) at (-3,7);
\coordinate [label= above:SO(3)]  (A) at (0,7);
\coordinate [label= above:$\Sigma_4$] (B) at (2,7) ;
\coordinate [label= above:$A_4$] (C) at (4,7) ;
\coordinate [label= above:$A_5$] (D) at (6,7) ;  
\coordinate [label= above:$D_4$] (E) at (8,7) ;

\fill[black] (A) circle (2pt);
\fill[black] (B) circle (2pt);
\fill[black] (C) circle (2pt);
\fill[black] (D) circle (2pt);
\fill[black] (E) circle (2pt);

\coordinate [label=$\cT$]  (FF) at (-3,5.5);
\coordinate[label=above:$T$] (M) at (0,5.5);

\coordinate [label=$\cD$]  (FG) at (-3,4);
\coordinate [label= above:$D_{6}$]  (H) at (5,4);
\coordinate [label= above:$D_{8}$]  (I) at (6,4);
\coordinate [label= above:$D_{10}$]  (L) at (6.6,4);
\coordinate[label=$...$] (J) at (7.3,3.86);
\coordinate[label=above:$O(2)$] (K) at (8,4);

\fill[black] (H) circle (2pt);
\fill[black] (I) circle (2pt);
\fill[black] (K) circle (2pt);
\fill[black] (L) circle (2pt);
\fill[black] (M) circle (2pt);

\end{tikzpicture}
\end{center}
 
\ \\

The topology on $\cE$ is discrete, $\cT$ consists of one point $T$ and $\cD$ forms a sequence of points converging to $O(2)$.

Before we go any further we recall the space $\cF(O(2))/O(2)$. It consists of two parts: toral and dihedral. To distinguish between these parts and their analogues for $SO(3)$ we choose the notation $\tcT$ for the toral part of $O(2)$ and $\tcD$ for the dihedral part of $O(2)$ (Note, that in \cite{BarnesO(2)new} the notation without tilda was used for the toral and dihedral parts of $O(2)$). We will stick to this new notation convention throughout the paper. The toral part is just one point $T$ corresponding to the maximal torus and all its subgroups. The dihedral part corresponds to all dihedral subgroups together with $O(2)$ and we present it below.
 
\begin{center}
\begin{tikzpicture}[dot/.style={circle, inner sep=1pt,fill,label={#1},name=#1,color=black}]
\usetikzlibrary{calc}
\coordinate[label=$\mathrm{Space}\ \cF(O(2))/O(2)$] (O) at (3,6);
\coordinate[label=$\mathrm{Part\ (Subspace)}$] (O) at (-3,6);

\coordinate [label=$\tcT$]  (FF) at (-3,5.5);
\coordinate[label=above:$T$] (M) at (0,5.5);

\coordinate [label=$\tcD$]  (FG) at (-3,4);
\coordinate [label= above:$D_2$]  (F) at (0,4);
\coordinate [label= above:$D_4$]  (G) at (3,4);
\coordinate [label= above:$D_{6}$]  (H) at (5,4);
\coordinate [label= above:$D_{8}$]  (I) at (6,4);
\coordinate [label= above:$D_{10}$]  (L) at (6.6,4);
\coordinate[label=$...$] (J) at (7.3,3.86);
\coordinate[label=above:$O(2)$] (K) at (8,4);

\fill[black] (F) circle (2pt);
\fill[black] (G) circle (2pt);
\fill[black] (H) circle (2pt);
\fill[black] (I) circle (2pt);
\fill[black] (K) circle (2pt);
\fill[black] (L) circle (2pt);
\fill[black] (M) circle (2pt);

\end{tikzpicture}
\end{center}
 
Notice that the only difference in the dihedral parts for $O(2)$ and $SO(3)$ is captured by the fact that the dihedral part for $O(2)$ is a disjoint union of $\cD$ and two points (corresponding to $D_2$ and $D_4$ respectively). 
At a first glance the toral part for $SO(3)$ looks the same as the toral part for $O(2)$. However, for $SO(3)$ it contains information about $D_2$ (since $D_2$ is conjugate to $C_2$ in $SO(3)$), whereas for $O(2)$ it does not. These differences will become significant in Section \ref{ideBurnside}.


\subsection{Left Bousfield localisation and cellularisation}\label{sec:BousfieldAndCell}
In this section we briefly recall two ways of changing a given stable model structure: left Bousfield localisation at an object and cellularisation. We will repeatedly use them in the rest of the paper.

\subsubsection{Left Bousfield localisation at an object.}

For details on left Bousfield localisation at an object we refer the reader to \cite[Section IV.6]{MandellMay}. We recall the following result, which is \cite[Chapter IV, Theorem 6.3]{MandellMay}.

\begin{thm}\label{thm:LBLatE_MM}Suppose $E$ is a cofibrant object in $G\endash\mathrm{Sp}$ or a cofibrant based $G$--space. Then there exists a new model structure on the category $G\endash\mathrm{Sp}$, where a map $f:X \lra Y$ is
\begin{itemize}[noitemsep]
\item a weak equivalence if it is an $E$--equivalence, i.e. $Id_E\wedge f: E\wedge X \lra E\wedge Y$ is a weak equivalence
\item cofibration if it is a cofibration with respect to the stable model structure
\item fibration if it has the right lifting property with respect to all trivial cofibrations.
\end{itemize}
The $E$--fibrant objects $Z$ are the $E$--local objects, i.e. such that $[f, Z]^G: [Y,Z]^G \lra [X,Z]^G$ is an isomorphism for all $E$--equivalences $f$.  
$E$--fibrant approximation gives Bousfield localisation $\lambda : X\lra L_EX$ of $X$ at $E$. 
\end{thm}
 
We use the notation $L_E(G\endash\mathrm{Sp})$ for the model category described above and will refer to it as a \textbf{left Bousfield localisation of the category of  $G$--spectra at $E$}. Notice that if $E$ and $F$ are cofibrant objects in $G\endash\mathrm{Sp}$  then the localisation first at $E$ and then at $F$ is the same model category as the localisation at $E\wedge F$ (and $F\wedge E$).

Recall that, an $E$--equivalence between $E$--local objects is a weak equivalence (see \cite[Theorems 3.2.13 and 3.2.14]{Hirschhorn}).

In this paper $X \in G\endash\mathrm{Sp}$ is usually of a form $eS_{\bQ}$ where $e$ is an idempotent of a rational Burnside ring $\mA(G)_\bQ$ and $S_{\bQ}$ is a rational sphere spectrum (see \cite[Section 5]{BarnesFiniteG} for construction of the rational sphere spectrum $S_\bQ$). Since we use idempotents of rational Burnside ring all our localisations are smashing (see \cite{Ravenel_Localization} for definition of a smashing localisation). Thus they preserve homotopically compact generators, see Definition \ref{compactobj}, since the fibrant replacement preserves infinite coproducts.

\subsubsection{Cellularisation.}\label{sec:cellularisation}
A cellularisation of a model category
is a right Bousfield localisation at a set of objects.
Such a localisation exists by \cite[Theorem 5.1.1]{Hirschhorn}
whenever the model category is right proper and cellular. 
When we are in a stable context the results of \cite{BarnesConstanze}
can be used.

In this section we recall the notion of cellularisation when $\cC$ is a stable model category and some of basic definitions and results.

\begin{defn}
Let $\cC$ be a stable model category and $K$ a stable set of objects of $\cC$, i.e. a set such that a class of $K$--cellular objects of $\cC$ is closed under desuspension (Note that the class is always closed under suspension). We call $K$ a set of \textbf{cells}.
We say that a map $f : A \longrightarrow B$ of $\cC$ is a \textbf{$K$--cellular equivalence} if
the induced map
\[
[k,f]^\cC_*: [k,A]^\cC_* \longrightarrow [k,B]^\cC_*
\]
is an isomorphism of graded abelian groups for each $k \in K$. An object $Z \in \cC$ is said to be
\textbf{$K$--cellular} if
\[
[Z,f]^\cC_*: [Z,A]^\cC_* \longrightarrow [Z,B]^\cC_*
\]
is an isomorphism of graded abelian groups for any $K$--cellular equivalence $f$.
\end{defn}

\begin{defn}
A \textbf{right Bousfield localisation} or \textbf{cellularisation} of $\cC$ with respect to
a set of objects $K$ is a model structure $K \cell \cC$ on $\cC$ such that
\begin{itemize}[noitemsep]
\item the weak equivalences are $K$--cellular equivalences
\item the fibrations of $K \cell \cC$ are the fibrations of $\cC$
\item the cofibrations of $K \cell \cC$ are defined via left lifting property.
\end{itemize}
\end{defn}

By \cite[Theorem 5.1.1]{Hirschhorn}, if $\cC$ is a right proper, cellular model category
and $K$ a set of objects in $\cC$, then the cellularisation of $\cC$ with respect to $K$, $K \cell \cC$,
exists and is a right proper model category.
The cofibrant objects of $K \cell \cC$
are called $K$--cofibrant and are precisely the
$K$--cellular and cofibrant objects of $\cC$.

The cellularisation of a right proper, cellular, stable model category at a stable set of cofibrant objects $K$ is very well behaved (see \cite[Theorem 5.9]{BarnesConstanze}), in particular it is proper, cellular and stable.

There is another important property we will often want the cells to satisfy,
which makes right localisation behave in an even more tractable manner,
see \cite[Section 9]{BarnesConstanze}. This property is variously called
small, compact or finite. We chose to call it \emph{homotopically compact}, since there are several different meanings of compactness in the literature.

\begin{defn}\cite[Definition 2.1.2]{SchwedeShipleyMorita}\label{def:hocompact}
An object $X$ in a stable model category $\cC$ is  \textbf{homotopically compact} if for any family of objects $\{A_i\}_{i \in I}$ the canonical map  
$$\bigoplus_{i\in I}[X,A_i]^{\cC} \lra [X, \coprod_{i \in I}A_i]^{\cC}$$ is an isomorphism in the homotopy category of $\cC$.
\end{defn}

Recall that a homotopy category of a stable model category is triangulated \cite[Definition 7.1.1]{Hovey}. In this setting we can make the following definition after \cite[Definition 2.1.2]{SchwedeShipleyMorita}.

\begin{defn}\label{compactobj}
Let $\cC$ be a triangulated category with infinite coproducts. A full triangulated subcategory of $\cC$ (with shift and triangles induced from $\cC$) is called \textbf{localising} if it is closed under coproducts in $\cC$. A set $\cP$ of objects of $\cC$ is called a  \textbf{set of generators} if the only localising subcategory of $\cC$ containing objects of $\cP$ is the whole of $\cC$. 
An object of a stable model category is called a generator if it is so when considered as an object of the homotopy category.
\end{defn}

Using \cite[Lemma 2.2.1]{SchwedeShipleyMorita}
it is routine to check that if $K$ consists of
homotopically compact objects
of $\cC$ then $K$ is a set of generators for $K \cell \cC$.
Hence we know a set of generators for each of our cellularisations.

Notice that derived functors of both left and right Quillen equivalences preserve homotopically compact objects.

\subsection{Idempotents, splitting and reductions}\label{ideBurnside}

By the results of tom Dieck \cite[5.6.4, 5.9.13]{tomDieck} there is an isomorphism of rings $$\A(SO(3))_\bQ =C(\cF(SO(3))/SO(3),\bQ).$$  Here $\A(SO(3))_\bQ$ is the rational Burnside ring for $SO(3)$ and $C(\cF(SO(3))/SO(3),\bQ)$ denotes the ring of continuous functions on the orbit space $\cF(SO(3))/SO(3)$ with values in discrete space $\bQ$. 

Thus it is clear that idempotents of the rational Burnside ring of $SO(3)$ correspond to the characteristic functions on subspaces of the orbit space  $\cF(SO(3))/SO(3)$ discussed in Section \ref{subgroups} which are both open and closed.  

In this paper we use the following idempotents in the rational Burnside ring of $SO(3)$: $e_{\cT}$ corresponding to the characteristic function of the toral part $\cT$, i.e. the conjugacy class of the torus $T$, $e_{\cD}$ corresponding to the characteristic function of the dihedral part $\cD$ and $e_{\cE}$ corresponding to the characteristic function of the exceptional part $\cE$. Since $\cE$ is a disjoint union of 5 points, it is in fact a sum of 5 idempotents, one for every (conjugacy class of a) subgroup in the exceptional part: $e_{SO(3)}$, $e_{\Sigma_4}$, $e_{A_4}$, $e_{A_5}$ and $e_{D_4}$. We use a simplified notation $e_H$ to mean $e_{(H)_{SO(3)}}$ here.
 
Analogously, we will use the notation $e_{\tcT}$ for the idempotent in the rational Burnside ring of $O(2)$ corresponding to the toral part $\tcT$ and $e_{\tcD}$ for the idempotent corresponding to the dihedral part $\tcD$ of $O(2)$.

For an idempotent $e \in \A(SO(3))_\bQ$ and a rational sphere spectrum $S_\bQ$ (see \cite[Section 5]{BarnesFiniteG} for construction) we define $eS_\bQ$ to be the homotopy colimit (a mapping telescope) of the diagram 
\[
\xymatrix{S_\bQ \ar[r]^{e} & S_\bQ\ar[r]^{e} & S_\bQ\ar[r]^{e} &...\ .}
\]
We ask for this spectrum to be cofibrant either by choosing a good construction of homotopy colimit, or by cofibrantly replacing the result in the stable model structure for $SO(3)$--spectra. Now, by \cite[Chapter IV, Theorem 6.3]{MandellMay} (see also Theorem \ref{thm:LBLatE_MM}) the following left Bousfield localisations exist: $L_{e_{\cT}S_\bQ}(SO(3)\endash\mathrm{Sp})$, $L_{e_{\cD}S_\bQ}(SO(3)\endash\mathrm{Sp})$, $L_{e_{\cE}S_\bQ}(SO(3)\endash\mathrm{Sp})$. Also,   $L_{e_{H}S_\bQ}(SO(3)\endash\mathrm{Sp})$ exists for any exceptional subgroup $H \in \cE$.

The first step on the way towards an algebraic model for rational $SO(3)$--spectra is to split this category using above idempotents of the Burnside ring $\A(SO(3))_\bQ$. By 
\cite[Theorem 4.4]{Barnes_Splitting} we get the following decomposition.

\begin{prop}\label{prop:splittingSO(3)}The following adjunction
\[
\xymatrix@R=2pc{
 SO(3) \endash\mathrm{Sp}_\bQ\ \ar@<-1ex>[d]_{\triangle} \\ L_{e_{\cT}S_\bQ}(SO(3)\endash\mathrm{Sp})\times L_{e_{\cD}S_\bQ}(SO(3)\endash\mathrm{Sp})\times L_{e_{\cE}S_\bQ}(SO(3)\endash\mathrm{Sp}) \ar@<-0.5ex>[u]_{\Pi}
}
\]
is a strong monoidal Quillen equivalence, where $SO(3) \endash\mathrm{Sp}_{\bQ}$ denotes the category of rational $SO(3)$ orthogonal spectra, the left adjoint is the diagonal functor and the right adjoint is the product.
\end{prop}

The main idea is to relate each of these localised categories to corresponding ones for simpler groups. Thus we recall, that an inclusion $i: H \lra G$ of a subgroup $H$ into a group $G$ induces two adjoint pairs at the level of orthogonal spectra: induction--restriction and restriction--coinduction, see \cite[Section V.2]{MandellMay}.

\[
\xymatrix@C=6pc{
G\endash\mathrm{Sp}\ 
\ar@<-0ex>[r]|-(.5){\ i^*\ }
&
N\endash\mathrm{Sp}\  
\ar@/^1pc/[l]^(.5){F_H(G_+,-)}
\ar@/_1pc/[l]_(.5){G_+\wedge_H -}
}
\]

These are both Quillen pairs with respect to the usual stable model structures on both sides. On the way to obtain an algebraic model for rational $SO(3)$--spectra we will relate both the toral and dihedral parts of this category to the corresponding parts for rational $O(2)$--spectra. The natural choice of adjunction between $SO(3)$--spectra and $O(2)$--spectra would be the induction and restriction functors. However, this turns out not to be a Quillen adjunction between the toral parts, as we discuss below.

\begin{prop}\label{notQErest_ind}Suppose $e_{\cT}$ is the idempotent in $\A(SO(3))_\bQ$ corresponding to the characteristic function of the toral part $\cT$ (i.e. all subconjugates of the maximal torus of $SO(3)$) and $e_{\tcT}$ is the idempotent in $\A(O(2))_\bQ$ corresponding to the characteristic function of the toral part $\tcT$ (i.e. all  subconjugates of the maximal torus of $O(2)$). Then
\[
\xymatrix{
i^\ast\ :\ L_{e_{\cT}S_\bQ}(SO(3)\endash\mathrm{Sp})\ \ar@<-1ex>[r] & L_{e_{\tcT}S_\bQ}(O(2)\endash\mathrm{Sp})\ :\ SO(3)_+\wedge_{O(2)}- \ar@<-0.5ex>[l]
}
\]
is not a Quillen adjunction.
\end{prop}
\begin{proof}It is enough to show that $SO(3)_+\wedge_{O(2)}-$ does not preserve acyclic cofibrations.  This argument is the same as the one in \cite[Proposition 4.5]{KedziorekExceptional}, since $D_2$ is conjugate to $C_2$ in $SO(3)$ and thus $i^*(e_{\cT})\neq e_{\tcT}$.  
\end{proof}

Although the adjunction above does not behave well with respect to these model structures, the one with restriction and coinduction does, as is shown in Proposition \ref{cyclicAdj} below.

\begin{prop}\label{dihRightAdj}Suppose $e_{\cD}$ is the idempotent of $\A(SO(3))_\bQ$ corresponding to all dihedral supgroups of order greater than $4$ and all subgroups isomorphic to $O(2)$. Then
\[
\xymatrix{
i^\ast\ :\ L_{e_{\cD}S_\bQ}(SO(3)\endash\mathrm{Sp})\ \ar@<-1ex>[r] & L_{i^*(e_{\cD})S_\bQ}(O(2)\endash\mathrm{Sp})\ :\ SO(3)_+\wedge_{O(2)}- \ar@<-0.5ex>[l]
}
\]
is a Quillen adjunction.
\end{prop}
\begin{proof}The proof follows the same pattern as the proof of \cite[Proposition 4.4]{KedziorekExceptional}. It was a Quillen adjunction before localisation by \cite[Chapter V, Proposition 2.3]{MandellMay} so the left adjoint preserves cofibrations. It preserves acyclic cofibrations as $SO(3)_+\wedge_{O(2)}-$ preserved acyclic cofibrations before localisation and we have a natural (in $O(2)$--spectra $X$) isomorphism: 
$$(SO(3)_+\wedge_{O(2)} X) \wedge e_{\cD}S_\bQ \cong SO(3)_+ \wedge_{O(2)} (X \wedge i^*(e_{\cD}S_\bQ))$$  \end{proof}

It turns out that the other adjunction -- restriction and coinduction adjunction -- gives a Quillen pair under general conditions on localisations
\begin{lem}\label{QuillenAdjForIandH}\cite[Lemma 4.6]{KedziorekExceptional} Suppose $G$ is any compact Lie group, $i: H \lra G$ is an inclusion of a subgroup and $V$ is an open and closed set in $\cF(G)/G$. Then the adjunction 
\[
\xymatrix{
i^\ast\ :\ L_{e_VS_\bQ}(G\endash\mathrm{Sp})\ \ar@<+1ex>[r] & L_{i^\ast (e_{V})S_\bQ}(H\endash\mathrm{Sp})\ :\ F_H(G_+,-) \ar@<+0.5ex>[l]
}
\]
is a Quillen pair. We use notation $e_V$ here for the idempotent corresponding to the characteristic function on $V$.
\end{lem}

In the next sections we will repeatedly use the above lemma, mainly in situations where after a further localisation of the right hand side we get a Quillen equivalence. To prepare for that, we distinguish the following two cases.

\begin{cor}\label{localisedQAdjunctions}
Let $\cD$ denote the dihedral part of $SO(3)$ and $e_{\cD}$ the corresponding idempotent. Let $i: O(2) \lra SO(3)$ be an inclusion. Then 
\[
\xymatrix{
i^\ast\ :\ L_{e_{\cD}S_\bQ}(SO(3)\endash\mathrm{Sp})\ \ar@<+1ex>[r] & L_{i^*(e_{\cD})S_\bQ}(O(2)\endash\mathrm{Sp})\ :\ F_{O(2)}(SO(3)_+,-) \ar@<+0.5ex>[l]
}
\]
is a Quillen adjunction. 
\end{cor}
\begin{rmk}\label{rmk:different_localisations} Note that the idempotent on the right hand side $i^*(e_{\cD})$ corresponds to the dihedral part of $O(2)$ excluding all subgroups $D_2$ and $D_4$. Thus $i^*(e_{\cD}) = i^*(e_{\cD}) e_{\tcD}$. 
\end{rmk}

\begin{prop}\label{cyclicAdj}Let $i: O(2) \lra SO(3)$ be an inclusion. Then the following adjunction
\[
\xymatrix{
i^\ast \ :\ L_{e_{\cT}S_\bQ}(SO(3)\endash\mathrm{Sp})\  \ar@<+1ex>[r] & \ L_{e_{\tcT}S_\bQ}(O(2)\endash\mathrm{Sp}) \ :\ F_{O(2)}(SO(3)_+,-) \ar@<+0.5ex>[l]
}
\]
is a strong monoidal Quillen adjunction, where the idempotent on the right hand side corresponds to the family of all subgroups of $O(2)$ subconjugate to a maximal torus $SO(2)$ in $O(2)$.
\end{prop}
\begin{proof}This follows from Lemma \ref{QuillenAdjForIandH} and the composition of Quillen adjunctions:
{\small \[
\xymatrix@C=3pc{
L_{e_{\cT}S_\bQ}(SO(3)\endash\mathrm{Sp})\ 
\ar@<+1ex>[r]^{i^\ast}
&
\ L_{i^\ast(e_{\cT}) S_{\bQ}}(O(2)\endash\mathrm{Sp})\ 
\ar@<+0.5ex>[l]^{F_{O(2)}(SO(3)_+,-)}
\ar@<+1ex>[r]^(.54){\Id}
&
\ L_{e_{\tcT}S_\bQ}(O(2)\endash\mathrm{Sp})
\ar@<+0.5ex>[l]^(.46){\Id}
}
\]}
Note that $i^\ast(e_{\cT} S_{\bQ})$ has non-trivial geometric fixed points not only for all cyclic subgroups of $O(2)$ and $SO(2)$, but also for $D_2$, as $D_2$ is conjugate to $C_2$ in $SO(3)$. To ignore that and take into account only toral part we use the fact that $e_{\tcT} i^*(e_\cT)=e_{\tcT}$, which implies that the identity adjunction above is a Quillen pair.
\end{proof}

\section{The toral part}\label{cyclicPart}

In this section we use results from \cite{BGKS} and \cite{BarnesO(2)new} to obtain an algebraic model for the toral part of rational $SO(3)$--spectra. The first paper establishes a zig-zag of symmetric monoidal Quillen equivalences between rational $SO(2)$--spectra, while the second one lifts this comparison to one compatible with the $W=O(2)/SO(2)$--action to obtain an algebraic model for the toral part of rational $O(2)$--spectra. 
 
We begin by describing the category $d\cA(O(2),\tcT)$ in Section \ref{A(G,c)} and $d\cA(SO(3),\cT)$ in Section \ref{The category A(SO(3),c)}. Then we proceed to establishing the comparison between the toral part of rational $SO(3)$--orthogonal spectra and its algebraic model, $d\cA(SO(3),\cT)$.

\subsection{Categories $\cA(O(2),\tcT)$ and $d\cA(O(2),\tcT)$}\label{A(G,c)}\label{subsectionA(O(2),c)}
Before we are ready to describe the category  $\cA(SO(3),\cT)$ we have to introduce the category $\cA(O(2),\tcT)$. We give a short description of $\cA(O(2),\tcT)$ as a category on the objects of $\cA(SO(2))$ with $W$--action. Recall that $W=O(2)/SO(2)$ is the group of order 2.

Material in this section is based on \cite{GreenleesS1} and \cite[Section 3]{BarnesO(2)new}. 

\begin{defn}\label{euler_classes1}Let $\cF$ denote the family of all finite cyclic subgroups in $O(2)$. Then we define a ring in graded $\bQ[W]$--modules $$\cO_{\cF}:=\prod_{H\in \cF}\bQ[c_H]$$ where each $c_H$ has degree $-2$ and $w$ (the non-trivial element of $W$) acts on each $c_H$ by $-1$. For simplicity we set $c:=c_1$.

We use the notation $\cE^{-1}\cO_\cF$ for the following colimit of localisations:
$$\colim_{k}\cO_{\cF}[c^{-1}, c_{C_2}^{-1},...,c_{C_{k}}^{-1}]$$
where the maps in the colimit are the inclusions. $\cE^{-1}\cO_\cF$ is an $\cO_\cF$--module using the inclusion $$\cO_\cF \lra \cE^{-1}\cO_\cF.$$
\end{defn}
Notice that we can perform a similar construction on the ring $\widetilde{\cO_{\cF}}:=(1-e_1)\cO_{\cF}$ and call it $\widetilde{\cE^{-1}\cO_{\cF}}$, where $e_1$ is the projection on the first factor in the ring $\cO_{\cF}$. Then another way to define  $\cE^{-1}\cO_\cF$ is as $\bQ[c,c^{-1}]\times \widetilde{\cE^{-1}\cO_{\cF}}$. This last description of $\cE^{-1}\cO_\cF$ will be useful when we compare this model to the one for toral part of rational $SO(3)$--spectra.

\begin{defn}\label{defn:object_in_AO(2)}
An object of $\cA(O(2),\tcT)$ consists of a triple $(M,V,\beta)$ where $M$ is an $\cO_{\cF}$--module in $\bQ[W]$--modules, $V$ is a graded rational vector space with a $W$--action and $\beta$ is a map of $\cO_{\cF}$--modules (in $\bQ[W]$--modules)
$$\beta : M \lra \cE^{-1}\cO_{\cF}\otimes V$$
 such that \\
($\star$) $\cE^{-1}\cO_\cF\otimes_{\cO_{\cF}}\beta$ is an isomorphism of $\cE^{-1}\cO_{\cF}$--modules in $\bQ[W]$--modules.

A morphism between two such objects $(\alpha, \phi): (M,V,\beta) \lra (M',V',\beta ')$ consists of a map of $\cO_{\cF}$--modules $\alpha : M \lra M'$ and a map of graded $\bQ[W]$--modules such that the relevant square commutes.
\end{defn}
Notice that instead of modules over $\cO_{\cF}$ in $\bQ[W]$--modules we can consider modules over $\cO_{\cF}[W]$ in $\bQ$--modules, where $\cO_{\cF}[W]$ is a group ring with a twisted $W$--action (namely $wc_H=-c_Hw$). We will use this description in the next section. Similarly, $\cE^{-1}\cO_{\cF}[W]$ denotes a group ring with a twisted $W$--action.

\begin{defn}
An object of $d\cA(O(2),\tcT)$ is an object of $\cA(O(2),\tcT)$ equipped with a differential, or in other words it consists of a triple $(M,V,\beta)$ where $M$ is an $\cO_{\cF}$--module in $\Ch(\bQ[W])$, $V$ is an object of $\Ch(\bQ[W])$ and $\beta$ is a map of $\cO_{\cF}$--modules (in $\Ch(\bQ[W])$)
$$\beta : M \lra \cE^{-1}\cO_{\cF}\otimes V$$
 such that \\
 ($\star$) $\cE^{-1}\cO_\cF\otimes_{\cO_{\cF}}\beta$ is an isomorphism of $\cE^{-1}\cO_{\cF}$--modules in $\Ch(\bQ[W])$. \\
 A morphism in this category is a morphism in $\cA(O(2),\tcT)$ which commutes with the differentials. 
\end{defn}

We proceed to discussing the properties of  the category  $d\cA(O(2),\tcT)$.
Firstly, all limits and colimits exist in  $d\cA(O(2),\tcT)$, by an argument analogous to \cite[Definition 2.2.1]{BGKS}.

The existence of a model structure on $d\cA(O(2),\tcT)$ follows from \cite[Appendix A]{GreenleesS1}

\begin{thm} There is a stable, proper model structure on the category  $d\cA(O(2),\tcT)$ where the weak equivalences are the homology isomorphisms. The cofibrations are the injections and the fibrations are defined via the right lifting property. We call this model structure the \textbf{injective model structure}.
\end{thm}
 The existence of another, monoidal model category structure on $d(\cA(O(2), \tcT))$ was established in \cite{BarnesO(2)new}. However, since we are not considering monoidality of the algebraic model in this paper, the injective model structure on $d\cA(O(2),\tcT)$ is enough for our purposes.

\subsection{Categories $\cA(SO(3),\cT)$ and  $d\cA(SO(3),\cT)$ }\label{The category A(SO(3),c)}
Looking at the toral parts of the spaces of subgroups of $SO(3)$ and $O(2)$ we see that the stabiliser of the trivial subgroup is connected in $SO(3)$, while it is not in $O(2)$. This is a consequence of the fact that the maximal torus is not normal in $SO(3)$ and it is the main ingredient capturing the difference between the algebraic model for the toral part of rational $SO(3)$--spectra and the toral part of rational $O(2)$--spectra.  

Let us denote by $\cF_{SO(3)}$ the family of all finite cyclic subgroups in $SO(3)$. Then we use the simplified notation $\cO_{\oF}:=\cO_{\cF_{SO(3)}}$, by which we mean a graded ring $\bQ[d]\times \prod_{(H)\in \cF_{SO(3)}, H \neq 1}\bQ[c_{(H)}]$ where $d$ is in degree $-4$ and all $c_{(H)}$ are in degree $-2$. The non-trivial element $w \in W$ acts on it by fixing $d$ and sending $c_{(H)}$ to $-c_{(H)}$ for all subgroups $H\in \oF_{SO(3)}$, $H\neq 1$. 

We define the ring $\cO_{\oF}[W]$ as a product of $\bQ[d]$ (with trivial $W$-action) and a group ring $(1-e_1)\cO_{\oF}[W]$ with the twisted $W$--action, that is $wc_{(H)}=-c_{(H)}w$ for $H\in \cF_{SO(3)}$, $H\neq1$. 

Recall that $c$ was the element of the first factor of the ring $\cO_\cF$ (see Definition \ref{euler_classes1}). 
There is an adjunction 
 \[
\xymatrix@C=5pc{\bQ\text{-mod} \ar@<+0.5em>[r]^{\text{Triv}}& \bQ[W]\text{-mod} \ar@<+0.5em>[l]^{(-)^W}}
\]
where $(\bQ[c])^W=\bQ[d]$ (recall that $\bQ[c]$ is the $\bQ[W]$--module with $W$--action given by $wc=-c$).
Thus using for example \cite[Section 2.3]{SchwShipEquiv} we have the adjunction 
 \[
\xymatrix@C=5pc{\bQ[d]\text{-mod in }\bQ\text{-mod} \ar@<+0.5em>[r]^{\bQ[c]\otimes_{\bQ[d]}-}& \bQ[c]\text{-mod in }\bQ[W]\text{-mod} \ar@<+0.5em>[l]^{(-)^W}}
\]

This extends to give the following result.
\begin{prop}\label{prop:firstadjal_mod}There is an adjunction:
\[
\xymatrix{
\cO_{\cF}\otimes_{\cO_{\oF}}- \ :\ \cO_{\oF}[W]\endash\mathrm{mod} \ar@<+1ex>[r] & \cO_{\cF}[W]\endash\mathrm{mod} \ :\ (-)^W\times \mathrm{Id} \ar@<+0.5ex>[l]
}
\]
\end{prop}
\begin{proof}The unit of this adjunction is the identity and the counit is the natural inclusion. 
\end{proof}

We can compose the above adjunction with the usual restriction--induction adjunction:
 \[
\xymatrix@C=5pc{\cO_{\cF}[W]\text{-mod} \ar@<+0.5em>[r]^{\cE^{-1}\cO_{\cF}\otimes_{\cO_{\cF}}-}& \cE^{-1}\cO_{\cF}[W]\text{mod} \ar@<+0.5em>[l]^{\text{res}}}
\]
to get the adjunction
\begin{equation}\label{eq:condition}
\xymatrix@C=5pc{\cO_{\oF}[W]\text{-mod} \ar@<+0.5em>[r]^{\cE^{-1}\cO_{\cF}\otimes_{\cO_{\oF}}-}& \cE^{-1}\cO_{\cF}[W]\text{mod} \ar@<+0.5em>[l]^{U}}
\end{equation}
in $\bQ$--modules.

We define the category $\cA(SO(3),\cT)$ as follows

\begin{defn}\label{def:objectA(SO(3),T)}
An object in $\cA(SO(3),\cT)$ consists of a triple $(M,V,\beta)$ where $M$ is an $\cO_{\oF}[W]$--module in $\bQ$--modules, $V$ is a graded rational vector space with a $W$--action and $\beta$ is a map of $\cO_{\oF}[W]$--modules
$$\beta : M \lra U(\cE^{-1}\cO_{\cF}\otimes V)$$
 such that the adjoint using (\ref{eq:condition}),  \\
($\star$) $\cE^{-1}\cO_{\cF} \otimes_{\cO_{\oF}}M \lra \cE^{-1}\cO_{\cF}\otimes V$ is an isomorphism of $\cE^{-1}\cO_{\cF}[W]$--modules. 

A morphism between two such objects $(\alpha, \phi): (M,V,\beta) \lra (M',V',\beta ')$ consists of a map of $\cO_{\oF}[W]$--modules $\alpha : M \lra M'$ and a map of graded $\bQ[W]$--modules such that the relevant square commutes.
\end{defn}
Notice that the condition on the map $\beta$ implies that the image of $e_1M$ must lie in ${( \bQ[c,c^{-1}]\otimes V)^W}$, i.e. in $W$--fixed points. From now on we will abuse the notation slightly and leave out the functor $U$ (\ref{eq:condition}) in the codomain of $\beta$ in $\cA(SO(3),\cT)$.

\begin{rmk}There are no idempotents in the category $\cA(SO(3),\cT)$, however the category of $\cO_{\oF}$--modules can be split, for example as $\bQ[d]\text{--mod}\times(1-e_1)\cO_{\oF}\text{--mod}$. We will use that property in the proof of Proposition \ref{prop:secondadjal_mod}.
\end{rmk}

\begin{defn}
An object of $d\cA(SO(3),\cT)$  consists of an $\cO_{\oF}[W]$--module $M$ equipped with a differential and a chain complex of $\bQ[W]$--modules $V$ together with a map of $\cO_{\oF}[W]$--modules
${\gamma : M \lra \cE^{-1}\cO_{\cF} \otimes V}$ which commutes with differentials. A differential on a $\cO_{\oF}[W]$--module $M$ consists of  maps $d_n: M_n \lra M_{n-1}$ such that $d_{n-1}\circ d_n = 0$ and $\bar{c}d_n =d_{n-2}\bar{c}$ where $\bar{c}$ consists of elements $c_{(H)}$ on the $H$--factor, for all $(H) \in\oF, H\neq 1$ and $0$ on the first factor and $\bar{d}d_n =d_{n-4}\bar{d}$ where $\bar{d}$ is $d$ on the first factor and $0$ everywhere else in the product.

A morphism in this category is a morphism in $\cA(SO(3),\cT)$ which commutes with the differentials. 
\end{defn}

We proceed to study the adjunction relating  $\cA(SO(3),\cT)$ and  $\cA(O(2),\tcT)$.

\begin{prop}\label{prop:secondadjal_mod}
We have the following adjunction, where the adjoints are defined in the proof:
\[
\xymatrix{
F \ :\  \cA(SO(3),\cT)\  \ar@<+1ex>[r] & \ \cA(O(2),\tcT) \ :\ R, \ar@<+0.5ex>[l]
}
\]
 \end{prop}
 \begin{proof}
Take $X=(\gamma : M \lra \cE^{-1}\cO_{\cF} \otimes V)$ in $d\cA(SO(3),\cT)$. Then $$F(X):= (\overline{\gamma}: \cO_{\cF} \otimes_{\cO_{\oF}}M \lra \cE^{-1}\cO_{\cF} \otimes V)$$
 where $\overline{\gamma}$ is the adjoint of $\gamma$ (since $ \cO_{\cF} \otimes_{\cO_{\oF}}-$ is a left adjoint from $\cO_{\oF}[W]$--modules to $\cO_{\cF}[W]$--modules, see Proposition \ref{prop:firstadjal_mod}). It is easy to see that this construction gives an object in $\cA(O(2),\tcT)$, i.e. that it satisfies the condition ($\star$) from Definition \ref{defn:object_in_AO(2)}. Since $\cE^{-1}\cO_{\cF} \otimes_{\cO_{\cF}} \overline{\gamma}$ is the same as $\cE^{-1}\cO_{\cF} \otimes_{\cO_{\oF}}\gamma$ in $\cE^{-1}\cO_{\cF}[W]$--modules and thus it is an isomorphism.

Now take $Y=(\delta : N \lra \cE^{-1}\cO_{\cF} \otimes U)$ in $d\cA(O(2),\tcT)$. Then $$R(Y):= (\delta \circ i: (e_1N)^W \times (1-e_1)N \lra N \lra \cE^{-1}\cO_{\cF} \otimes U)$$
where $i$ is the inclusion.

To see that $R(Y)\in \cA(SO(3),\cT)$ we show that the adjoint condition ($\star$) from Definition \ref{def:objectA(SO(3),T)} holds for $\delta \circ i$.

Thus we want to show that $$\overline{\delta \circ i}:\cE^{-1}\cO_{\cF}\otimes_{\cO_{\oF}}((e_1N)^W \times (1-e_1)N )\lra \cE^{-1}\cO_{\cF}\otimes U$$ is an isomorphism of $\cE^{-1}\cO_{\cF}[W]$ modules. 

Notice that we have a natural map $$\cE^{-1}\cO_{\cF}\otimes_{\cO_{\cF}}(\varepsilon_{N}): \cE^{-1}\cO_{\cF}\otimes_{\cO_{\oF}}((e_1N)^W \times (1-e_1)N ) \lra \cE^{-1}\cO_{\cF}\otimes_{\cO_{\cF}}(N )$$ where $\varepsilon$ is the counit of the adjunction from Proposition \ref{prop:firstadjal_mod}: 

After applying $e_1$ the map $e_1\varepsilon_N$ is an isomorphism for finitely generated modules $N$. Since every module is a colimit of finitely generated ones and $\otimes$ commute with colimits $e_1\varepsilon_N$ is an isomorphism for any $N$. Since $\varepsilon_N$ is an isomorphism away from $e_1$ it is an isomorphism.
To complete the argument notice that the following diagram commutes.
\[
\xymatrix @R=3pc @C=5pc{
\cE^{-1}\cO_{\cF}\otimes_{\cO_{\oF}}((e_1N)^W \times (1-e_1)N ) \ar[r]^(.6){\cE^{-1}\cO_{\cF}\otimes_{\cO_{\cF}}(\varepsilon_{N})} \ar[dr]_{\overline{\delta \circ i}} & \cE^{-1}\cO_{\cF}\otimes_{\cO_{\cF}}(N ) \ar[d]^{\overline{\delta}}\\
& \cE^{-1}\cO_{\cF}\otimes U
}
\]
where $\overline{\delta}$ is the adjoint of $\delta$ (see Proposition \ref{prop:firstadjal_mod}).

It is easy to see that this is an adjoint pair, since the unit is the identity and the counit is the pair of maps $(\varepsilon, \mathrm{Id})$ and the identity on graded $\bQ[W]$--modules. Here $\varepsilon$ is the counit of the adjunction in Proposition \ref{prop:firstadjal_mod}.
\end{proof}

\begin{prop}All small limits and colimits exist in  $\cA(SO(3),\cT)$. 
\end{prop}
\begin{proof} Suppose we have a diagram of objects $M_i \lra \cE^{-1}\cO_{\cF} \otimes V_i$ in $ \cA(SO(3),\cT)$ indexed by a category $I$. The colimit of this diagram is 
$$\colim_i M_i \lra \cE^{-1}\cO_{\cF} \otimes (\colim_i V_i).$$
If the diagram is finite, than the limit is formed in $ \cA(SO(3),\cT)$ in a similar way.
$$\lim_i M_i \lra \cE^{-1}\cO_{\cF} \otimes (\lim_i V_i).$$
To construct infinite limits in a category $\cA(SO(3),\cT)$ we use the same method as in \cite[Definition 2.2.1]{BGKS}. However, since we don't use the construction of infinite limits anywhere in this paper, we skip the technicalities. 

Verifying that these constructions define limits and colimits in $\cA(SO(3),\cT)$ is routine.
\end{proof}

Let $g\bQ[W]\endash\mathrm{mod}$ denote the category of graded $\bQ[W]$--modules and let $\mathrm{tors}\endash\cO_{\oF}[W]^f\endash \mathrm{mod}$ denote the category of $\oF$--finite torsion $\cO_{\oF}[W]$--modules. Recall that an $\cO_{\oF}[W]$--module $M$ is $\oF$--finite if it is a direct sum of its submodules $e_{(H)}M$:
$$M=\bigoplus_{(H)\in \oF} e_{(H)}M.$$

We define two functors relating $\cA(SO(3),\cT)$ to some simpler categories, which will allow us to create two clasess of injective objects in $\cA(SO(3),\cT)$. 
\begin{defn}\label{def:e_and_f}Let the functor $e: g\bQ[W]\endash \mathrm{mod} \lra  \cA(SO(3),\cT)$ be defined by
$$e(V):= (P \lra \cE^{-1}\cO_{\cF} \otimes V)$$
where $e_1P=\bQ[d,d^{-1}]\otimes V^+ \oplus \Sigma^2 \bQ[d,d^{-1}]\otimes V^-$ and $(1-e_1)P=(1-e_1)\cE^{-1}\cO_{\cF} \otimes V$. Here $V^+$ is the $W$--fixed part of $V$, $V^-$ is the $-1$ eigenspace and $\Sigma$ is the suspension.  The structure map is essentially just an inclusion.

Define a functor $f: \mathrm{tors}\endash\cO_{\oF}[W]^f\endash\mathrm{mod}\lra \cA(SO(3),\cT)$ by
$$f(N):= (N \lra 0).$$
\end{defn}
Notice that the domain for this functor was chosen, so that the $f(N) \in \cA(SO(3),\cT)$, i.e. it satisfies the condition ($\star$) from Definition \ref{def:objectA(SO(3),T)}.

\begin{prop}\label{adjointsInASO3C}For any object $X=(\gamma : M\lra \cE^{-1}\cO_{\cF}\otimes V)$ in $\cA=\cA(SO(3),\cT)$, any $V$ in  $\bQ[W]\endash\mathrm{mod}$ and any $N$ in $\mathrm{tors}\endash\cO_{\oF}[W]^f\endash\mathrm{mod}$ we have natural isomorphisms:
$$\hom_{\cA}(X, e(V))=\hom_{\bQ[W]}(U,V)$$
$$\hom_{\cA}(X, f(N))=\hom_{\cO_{\oF}[W]}(M,N)$$
\end{prop}

\begin{rmk}The proposition above implies that an object $e(V)$ is injective for any $V$ and that if $N$ is an injective  $\oF$--finite torsion $\cO_{\oF}[W]$--module then $f(N)$ is also injective.
\end{rmk}

\begin{lem}\label{lemmaForASS} The category $\cA(SO(3),\cT)$ is a (graded) abelian category of injective dimension $1$. Moreover it is split, i.e. every object $X$ of $\cA(SO(3),\cT)$ has a splitting $X=X_+ \oplus X_-$ so that $\hom(X_\delta, Y_\epsilon)=0$ and $Ext(X_\delta, Y_\epsilon)=0$ if $\delta\neq \epsilon$ and $(\Sigma X)_+ = \Sigma(X_{-})$ and $(\Sigma X)_- = \Sigma(X_{+})$.
\end{lem}
\begin{proof} The category $\cA(SO(3),\cT)$ is enriched in abelian groups and by construction of all limits and colimits we can conclude that it is an abelian category. 

For an object $X=(\gamma : M\lra \cE^{-1}\cO_{\cF}\otimes V)$ we construct the injective resolution of length 1 as follows. Let $TM:= \mathrm{ker}\gamma$, which is torsion, and thus there is an injective resolution of $\oF$--finite torsion $\cO_{\oF}[W]$--modules
\[
\xymatrix {
0 \ar[r] & TM \ar[r]& I' \ar[r] & J' \ar[r] & 0
}
\]
where $I',J'$ are injective $\oF$--finite torsion $\cO_{\oF}[W]$--modules, since $\bQ[d]$ and all $\bQ[c_{(H)}][W]$ are of injective dimension 1. 

Let us use simplified notation below. Let $P$ denote the $\cO_{\oF}[W]$--module from the definition of $e(V)$  (see Definition \ref{def:e_and_f}).\\
If $Q$ is the image of $\gamma$ then $J''= P/Q$ is divisible and an $\oF$--finite torsion $\cO_{\oF}[W]$--module and hence injective. We form a diagram of $\cO_{\oF}[W]$--modules
\[
\xymatrix {
& 0  \ar[d]& 0 \ar[d]& 0  \ar[d]&  \\
0 \ar[r] & TM \ar[r] \ar[d]& M \ar[r] \ar[d]& Q\ar[r] \ar[d]& 0 \\
0 \ar[r] & I' \ar[r] \ar[d] & I'\oplus P \ar[r] \ar[d] & P \ar[r] \ar[d] & 0\\
0 \ar[r] & J' \ar[r] \ar[d] & J'\oplus J'' \ar[r] \ar[d] & J'' \ar[r] \ar[d] & 0 \\
 & 0 & 0  & 0 &\\
}
\]

and hence a diagram:
\[
\xymatrix {
0 \ar[r] & M \ar[r] \ar[d]^{\gamma} & I'\oplus P \ar[r] \ar[d]& J'\oplus J''\ar[r] \ar[d]& 0 \\
0 \ar[r] & \cE^{-1}\cO_{\cF}\otimes V  \ar[r] & \cE^{-1}\cO_{\cF}\otimes V \ar[r] & 0 \ar[r]& 0\\
}
\]
which is the required resolution of $\gamma : M\lra \cE^{-1}\cO_{\cF}\otimes V$ in $\cA(SO(3),\cT)$.

Finally, the splitting is given by taking even--graded part and odd--graded part. This satisfies the required conditions since the resolution above of an object $X_\delta$ is entirely in parity $\delta$.
\end{proof}

\subsection{Model category $d\cA(SO(3),\cT)$}
In this section we will concentrate on the model category $d\cA(SO(3),\cT)$ and we will investigate its properties. First notice that all constructions from the previous section (limits and colimits, adjoints $F$ and $R$) pass naturally to the category $d\cA(SO(3),\cT)$. 

By the result of the previous section and \cite[Proposition 4.1.3]{GreenleesS1} we can construct the derived category of $\cA(SO(3),\cT)$ by taking objects with differential in $\cA(SO(3),\cT)$ and inverting the homology isomorphisms.  
\begin{thm}There is an injective model structure on the category $d\cA(SO(3),\cT)$ where weak equivalences are homology isomorphisms and cofibrations are monomorphisms.
\end{thm}
\begin{proof}Since the category $\cA(SO(3),\cT)$ is abelian of injective dimension 1 we can use the construction from \cite[Appendix A]{GreenleesS1}.  
\end{proof}
We call $d\cA(SO(3),\cT)$ with the injective model structure the \textbf{algebraic model for toral rational $SO(3)$--spectra}.

To show that the injective model structure is right proper in Proposition \ref{prop:rightproper_AT} we need to introduce a class of objects in $\cA(SO(3),\cT)$ called \emph{wide spheres}. This class generalises the images of representation spheres from rational $SO(3)$--spectra in $\cA(SO(3),\cT)$, hence the name.

\begin{defn}We define $\underline{c}^{2n}$ to be an element of $\cE^{-1}\cO_{\cF}$ of the form $(c^{2n}, c^{2n}, c^{2n},....)$. Notice, that we can view an element  $ \underline{c}^{2n}$ as an element of $\cO_{\oF}$ of the form $(d^n, c^{2n}, c^{2n},....)$.\\
We define $ \underline{c}^{2n+1}$ for $n>0$ to be an element of $\cE^{-1}\cO_{\cF}$ of the form $(c^{2n+1}, c^{2n+1}, c^{2n+1}, c^{2n+1},....)$.
\end{defn}

\begin{defn} A \textbf{wide sphere} in $\cA(SO(3),\cT)$ is an object $P=(S \lra \cE^{-1}\cO_{\cF} \otimes T)$ where $T$ is a graded $\bQ[W]$--module, which is finitely generated as a $\bQ$--module on elements $t_1,...,t_d$, where every $t_i$ is either $W$--fixed or $W$ acts on $t_i$ by $-1$ and $\mathrm{deg}(t_i)=k_i$. A module $S$ is an $\cO_{\oF}$--submodule of $\cE^{-1}\cO_{\cF} \otimes T$ generated by elements $\underline{c}^{a_i}\otimes t_1,..., \underline{c}^{a_d}\otimes t_d$ where ${a_i}$ is either even if $t_i$ is $W$--fixed or odd if $W$ acts on $t_i$ by $-1$, and an element $\sum_{i=1}^d\sigma_i \otimes t_i$ of $ \cE^{-1}\cO_{\cF} \otimes T$. It is also required that the structure map be the inclusion. We denote by $\cP$ the set of isomorphism classes of wide spheres.
\end{defn}

We want to show that there are enough wide spheres in $\cA(SO(3),\cT)$, i.e. for any $X \in \cA(SO(3),\cT)$ there exists an epimorphism from some coproduct of wide spheres to $X$.

\begin{prop}\label{xxxx}There are enough wide spheres in $\cA(SO(3),\cT)$.
\end{prop}
\begin{proof}We need to show that for any object $X=(\beta: N \lra \cE^{-1}\cO_{\cF} \otimes U)$ in $\cA(SO(3),\cT)$ and any $n\in N$ there exists a wide sphere $P$ and a map $P\lra X$ such that $n$ is in the image and for any  $u\in U$ there exists a wide sphere $\overline{P}$ and a map $\overline{P} \lra X$ such that $u$ is in the image. Since the adjoint of $\beta$ is an isomorphism it is enough to show the above condition for any $n\in N$.  

Take $X=(\beta: N \lra \cE^{-1}\cO_{\cF} \otimes U)$ in $\cA(SO(3),\cT)$ and $n\in N$. Then $\beta(n)=\sum_{i=1}^d \sigma_i \otimes t_i$. We may assume that for every $i$, either $t_i$ is $W$--fixed or $W$ acts on $t_i$ by $-1$. Then notice that since $e_1\beta(n)$ is $W$--fixed $e_1 \sigma_i$ will be of the form $c^{2k}$ if $t_i$ was $W$--fixed or $c^{2k+1}$ if $W$ acts on $t_i$ by $-1$ ($k$ is some integer here). 
 
For each $i$, there exist $p_i \in N$ such that $\beta(p_i)=\underline{c}^{2b_i} \otimes t_i$ if $t_i$ was $W$--fixed or $\beta(p_i)=\underline{c}^{2b_i+1} \otimes t_i$ if $W$ acts on $t_i$ by $-1$. Set $f= (\underline{c})^{2b_1+...+2b_d}$.
We may assume that the $b_i$ were large enough so that 
$\sigma_i \underline{c}^{2b_1+...+2b_d}/\underline{c}^{2b_i}$ is in $\cO_{\oF}$ if $t_i$ was $W$--fixed and 
$\sigma_i \underline{c}^{-1}\underline{c}^{2b_1+...+2b_d}/\underline{c}^{2b_i}$ is in $\cO_{\oF}$ if $W$ acts on $t_i$ by $-1$.

Now we have $$\beta(\sum^+ \sigma_i \underline{c}^{2b_1+...+2b_d}/\underline{c}^{2b_i}p_i + \sum^- \sigma_i \underline{c}^{-1}\underline{c}^{2b_1+...+2b_d}/\underline{c}^{2b_i}p_i)=\sum_{i=1}^d \sigma_i f \otimes t_i = \beta(fn)$$
where $\sum^+$ denotes the sum over all $t_i$ which are $W$--fixed and $\sum^-$ denotes the sum over all the others.

Since the adjoint of $\beta$ is an isomorphism there exists an element $\underline{c}^{2b}$ such that
  $$\underline{c}^{2b} (\sum^+ \sigma_i \underline{c}^{2b_1+...+2b_d}/\underline{c}^{2b_i}p_i + \sum^- \sigma_i \underline{c}^{-1}\underline{c}^{2b_1+...+2b_d}/\underline{c}^{2b_i}p_i)= \underline{c}^{2b} fn$$
  We take $\underline{c}^{2b}$ to be the smallest such element. 
  
We take a wide sphere $P=(S \lra \cE^{-1}_{\cO_{\cF}} \otimes T)$ where $T$ is a $\bQ$ vector space generated by $t_i$ for $i=1,...,d$, $\mathrm{deg}(t_i)=k_i$ and $S$ is an $\cO_{\oF}$ submodule of $ \cE^{-1}_{\cO_{\cF}} \otimes T$ generated by $\sum_{i=1}^d \sigma_i \otimes t_i$ and $\underline{c}^{2b} f \otimes t_i$ if $t_i$ is $W$--fixed and $\underline{c}^{2b-1} f \otimes t_i$ if $W$ acts on $t_i$ by $-1$. The structure map is the inclusion. 
  
To finish the proof we set a map from $P$ to $X$ by sending $\sum_{i=1}^d \sigma_i \otimes t_i$ to $n$ and $\underline{c}^{2b} f \otimes t_i$  to $\underline{c}^{2b} \underline{c}^{2b_1+...+2b_d}/\underline{c}^{2b_i}p_i$ if $t_i$ is $W$--fixed and $\underline{c}^{2b-1} f \otimes t_i$ to $\underline{c}^{2b-1} \underline{c}^{2b_1+...+2b_d}/\underline{c}^{2b_i}p_i$ if $W$ acts on $t_i$ by $-1$. 
  
The elements $\underline{c}^{2b}$ and $f$ are needed to ensure that the relation between $n$ and the $p_i$'s after applying $\beta$ is replicated in the wide sphere.  
\end{proof}

\begin{prop}\label{prop:rightproper_AT}The injective model structure on $d\cA(SO(3),\cT)$ is proper.
\end{prop}
\begin{proof}Since cofibrations are the monomorphism it is left proper. To show that it is right proper notice that among trivial cofibrations there are maps $0 \lra D^n \otimes P$, for any $P \in \cP$, where $D^n \otimes P$ denotes an object built from $P$ and $\Sigma P$ with the differential being the identity map from suspension of $P$ to $P$. Recall that $\cP$ denotes the set of isomorphism classes of wide spheres. Since there are enough wide spheres, the fibrations are in particular surjections. Right properness follows from the fact that in $\bQ[W]\endash\mathrm{mod}$ and $\cO_{\oF}[W]\endash\mathrm{mod}$ pullbacks along surjections of homology isomorphisms are homology isomorphisms.
\end{proof}

\begin{cor} The category $d\cA(SO(3),\cT)$ is a Grothendieck category.
\end{cor}
\begin{proof} Directed colimits are exact in $d\cA(SO(3),\cT)$, since they are in $R$--modules, for any ring $R$. Thus it remains to show that there is a (categorical) generator. We take $J:=\bigoplus_{P\in \cP} P$ where  $\cP$ is the set of all wide spheres. By Proposition \ref{xxxx} $\hom(J,-)$ is faithful and thus $J$ is a categorical generator.
\end{proof}

Next we define a set of objects which will be generators for the homotopy category of $d\cA(SO(3),\cT)$ with the injective model structure. Before we were considering categorical generators, but from now on the meaning of the word \emph{generator} is as in Definition \ref{compactobj}. Recall that if $\beta : M \lra \cE^{-1}\cO_{\cF}\otimes V$ is an object in $d\cA(SO(3),\cT)$, then  $M$ is in particular a module over $\cO_{\oF}[W]$ (which is an infinite product over conjugacy classes of cyclic subgroups in $SO(3)$, see beginning of Section \ref{The category A(SO(3),c)}).

\begin{defn}\label{generatorsforA(SO(3),c)} We define a set $\cK$ in $d\cA(SO(3),\cT)$ to consist of all suspensions and desuspensions of the following objects:
\begin{itemize}
\item for the trivial subgroup $$\sigma_1:=(\bQ_1 \lra 0),$$
where $\bQ$ is at the place indexed by the trivial subgroup and all other factors are $0$, 
\item for every $H \in \oF$, $H\neq 1$ $$\sigma_{H}:=(\bQ[W]_{(H)} \lra 0),$$
where $\bQ[W]$ is at the place indexed by the conjugacy class of a subgroup $H$ and all other factors are $0$ and
\item for the torus
$$\sigma_T:=(M \lra \cE^{-1}\cO_{\cF} \otimes \bQ[W]),$$
where $e_1 M= \bQ[d] \oplus \Sigma^2 \bQ[d]$, $(1-e_1)M= (1-e_1)\cO_{\cF}.$
Here the map is the inclusion.
\end{itemize}
\end{defn}

It remains to show that the set of cells $\cK$ is a set of generators for the injective model structure on $d\cA(SO(3),\cT)$.

\begin{thm}\label{generatoryDlaA(SO(3)cyklicznej)}The set $\cK$ is a set of homotopically compact generators for the category $d\cA(SO(3),\cT)$ with the injective model structure.
\end{thm}
\begin{proof} First note that $\sigma_T= (\cO_{\oF} \lra \cE^{-1}\cO_{\cF} \otimes \bQ) \oplus (N \lra \cE^{-1}\cO_{\cF}  \otimes \widetilde{\bQ})$, where $e_1 N= \Sigma^2 \bQ[d]$ and $(1-e_1)N=(1-e_1)\cO_\cF \otimes  \widetilde{\bQ}$ (here $\widetilde{\bQ}$ denotes $\bQ$ with the action of $w$ by $-1$). and both structure maps are inclusions.  We call the first summand $S^0$ and the second $\sigma_{T}^-$. Therefore it is enough to show that all suspensions and desuspensions of $\sigma_1,\ \sigma_{H}, \ \sigma_T^-,\ S^0$ for all $H \in \cF$, $H\neq 1$ form a set of generators. We will call this set $\cL$.  

All cells are homotopically compact since they are compact and fibrant replacement commutes with direct sums. 

We will show that if $[\sigma,X]^\cA_*=0$, for all $\sigma \in \cL$ then $H_*(X)=0$ and thus $X$ is weakly equivalent to $0$. By Lemma \ref{lemmaForASS}, \cite[Lemma 4.2.4]{GreenleesS1} and \cite[Theorem 3.8]{BarnesO(2)new} we can use the following Adams short exact sequence to calculate the maps in the derived category of $\cA= \cA(SO(3),\cT)$ from $X$ to $Y$ in $d\cA$:
\[
\xymatrix {
0 \ar[r] & \mathrm{Ext}_\cA(\Sigma H_*(X), H_*(Y))  \ar[r] & [X,Y]^\cA_* \ar[r] & {\hom}_\cA(H_*(X),H_*(Y)) \ar[r]& 0\\
}
\]

Observe that for every $X \in d\cA(SO(3),\cT)$, where $${X=(\gamma: P \lra \cE^{-1}\cO_\cF \otimes V)}$$ we have the following fibre sequence
\[
\xymatrix{
\widehat{X} \ar[r] & X \ar[r] &  e(V)
}
\]
where $e(V)$ is the functor described before Proposition \ref{adjointsInASO3C} and $\widehat{X}$ is the fibre of the map $X \lra e(V)$.

By definition, the structure map of $e(V)$ is an inclusion, and thus it is a torsion--free object.
To simplify the notation, let $$E\oF_+= (\Sigma^{-2} \bQ[d,d^{-1}]/\bQ[d] \lra 0) \oplus \bigoplus_{(H) \in \oF, H\neq 1}((\Sigma^{-2} \bQ[c_{(H)},c_{(H)}^{-1}]/\bQ[c_{(H)}])\lra 0) $$
We call the $H$--summand in the above formula $\alpha_H$. Then
$$\widehat{X}\simeq E\oF_+ \otimes X.$$

Now observe that every summand $\alpha_H$ in $E\oF_+$ is built as a sequential colimit from suspensions of $\alpha_H^n =(\bQ[c_{(H)}]/{c_{(H)}}^n \lra 0)$ and inclusions, or if it is the first summand $\alpha_1$ it is built as a sequential colimit of $\alpha_1^n=(\bQ[d]/d^n \lra 0)$ and inclusions,  and thus
 
$$[\sigma_K,\widehat{X}]^\cA_*=[\sigma_K,E\oF_+\otimes X]^\cA_*\cong [\sigma_K,  \bigoplus_{(H)} (\alpha_H \otimes X)]^\cA_*\cong \bigoplus_i [\sigma_K, \alpha_H \otimes X]^\cA_*$$
where the last isomorphism follows since $\sigma_K$ is a homotopically compact object. For all $H$, $\alpha^n_H$ is a strongly dualizable object (by \cite[Corollary 2.3.7 and Lemma 2.4.3]{GreenleesS1}), and thus we can proceed:
\begin{multline}[\sigma_K, \alpha_H \otimes X]^\cA_* \cong [\sigma_K, \colim_n \alpha^n_H \otimes X]^\cA_* \cong  \colim_i [\sigma_K,\hom(D(\alpha^n_H), X)]^\cA_* \cong \\ \colim_i [D(\alpha^n_H)\otimes\sigma_K, X]
\end{multline}
since $D(\alpha^n_H)\otimes \sigma_K =0$ if $K\neq H$ and every $D(\alpha^n_H)\otimes \sigma_H$ is finitely built from $\sigma_H$ and by assumption $[\sigma,X]=0$ for all $\sigma \in \cL$ we have that  $[D(\alpha^n_H)\otimes \sigma_H, X]=0$  and thus $[\sigma_H,\widehat{X}]^\cA_*=0$ for all $H \in \oF$.

Now take $X$ to be an object in $d\cA(SO(3),\cT)$ and assume that $[\sigma,X]^\cA_*=0$ for all $\sigma \in \cL$.
By the calculation above it follows that $[\sigma_H,\widehat{X}]^\cA_*=0$  for all $H \in \cF$.

From the Adams short exact sequence we get that 
$$ {\hom}_\cA(H_*(\sigma_H),H_*(\widehat{X}))=  {\hom}_\cA(\sigma_H,H_*(\widehat{X}))=e_{(H)}H_*(\widehat{X})=0$$

Since $H_*(\widehat{X}) = \bigoplus_{(H)\in \oF} e_H H_*(\widehat{X})$ we conclude that $\widehat{X}$ is weakly equivalent to $0$ and thus $[S^0,\widehat{X}]^\cA_*=0$ and $[\sigma_T^-,\widehat{X}]^\cA_*=0$. 

Now, by the fibre sequence and the fact that every fibre sequence induces a long exact sequence on $[E,-]$ we deduce that $[\sigma,e(V)]^\cA_*=0$ for every $\sigma \in \cL$. From the Adams short exact sequence it follows that
$$ {\hom}_\cA(H_*(S^0),H_*(e(V)))= {\hom}_\cA(S^0,H_*(e(V)))=H_*^+(e(V))=0$$
$$ {\hom}_\cA(H_*(\sigma_T^-),H_*(e(V)))=  {\hom}_\cA(\sigma_T^-,H_*(e(V))=H_*^-(e(V))=0$$
where $H_*^+(e(V))$ is the $W$ fixed part of $H_*(e(V))$ and $H_*^-(e(V))$ denotes $-1$ eigenspace. Since $H_*(e(V))=H_*^+(e(V))\oplus H_*^-(e(V))$ we get that $e(V)$ is weakly equivalent to 0.
Since the fibre sequence induces a long exact sequence in homology we conclude that $H_*(X)=0$ and thus $X$ is weakly equivalent to $0$ which finishes the proof.

\end{proof}

We finish this section by relating $d\cA(SO(3),\cT)$ and $d\cA(O(2),\tcT)$.

\begin{lem}\label{lem:before_cellFR}
The adjunction 
\[
\xymatrix@C=2pc{
\ d\cA(SO(3),\cT) \ 
\ar@<+1ex>[r]^{F}
&
\ d\cA(O(2),\tcT)\ 
\ar@<+0.5ex>[l]^{R}}
\]
is a Quillen pair when we equip both categories with the injective model structures. Recall that $F$ and $R$ are defined in the proof of Proposition \ref{prop:secondadjal_mod}.
\end{lem}
\begin{proof} The left adjoint is exact, so it preserves cofibrations (monomorphisms) and homology isomorphisms. \end{proof}

\begin{thm}\label{adjunction1} 
The adjunction

\[
\xymatrix@C=2pc{
d\cA(SO(3),\cT) \ 
\ar@<+1ex>[r]^(.4){F}
&
F(\cK)\endash\mathrm{cell}\endash d\cA(O(2),\tcT)\ 
\ar@<+0.5ex>[l]^(.6){R}
}
\]

is a Quillen equivalence, where $\cK$ is given in Definition \ref{generatorsforA(SO(3),c)}. 
\end{thm}
\begin{proof} 
We cellularise the left hand side of the adjunction in Lemma \ref{lem:before_cellFR} at the set $\cK$ and the right one at $F(\cK)$. The left hand side is then just $d\cA(SO(3),\cT)$ by Theorem \ref{generatoryDlaA(SO(3)cyklicznej)}. 
Thus to use the cellularisation principle \cite[Theorem 2.1]{GreenShipCell} we need to proof that the derived unit is an isomorphism for every element of $\cK$. Since the right adjoint preserves all weak equivalences it is enough to show that the categorical unit is a weak equivalence. However, we already know that the unit of this adjunction is the identity (it was shown in the proof of Proposition \ref{prop:secondadjal_mod}).
It remains to show that the elements of  the set $F(\cK)$ are homotopically compact in $d\cA(O(2),\tcT)$ with the injective model structure. This follows from the fact that $R$ preserves coproducts (notice that one component of $R$ is taking $W$--fixed points and over $\bQ$ this is isomorphic to taking $W$--orbits. The other components of $R$ are identities). This finishes the proof. 
\end{proof}

In the next section we will compare the cells coming from the topological generators (see Proposition \ref{homComCyclicGen}), with the ones used for cellularising $d\cA(O(2),\tcT)$. For these two set of cells to agree we now change the set of cells used for cellularising $d\cA(O(2),\tcT)$. We introduce the following self--Quillen equivalence (which is also an equivalence of categories) of $d\cA(O(2),\tcT)$ with the injective model structure. Use the notation $\widetilde{\bQ}$ for the $\bQ[W]$--module $\bQ$ with non--trivial $W$--action. We denote by $-\otimes \widetilde{\bQ}$ a self-adjoint on $d\cA(O(2),\tcT)$ defined as follows.
$$-\otimes \widetilde{\bQ}(\beta : M \lra \cE^{-1}\cO_{\cF}\otimes V):=(\beta \otimes \widetilde{\bQ}: M\otimes \widetilde{\bQ} \lra \cE^{-1}\cO_{\cF}\otimes (V\otimes \widetilde{\bQ}).$$
Thus below, we use the notation $\widetilde{F}$ to denote $-\otimes \widetilde{\bQ} \circ F$ and $\widetilde{R}$ to denote $R\circ -\otimes \widetilde{\bQ}$. 
 
The final result of this section follows from Theorem \ref{adjunction1}.
\begin{cor}\label{koncoweCor} The following is a Quillen equivalence, where $\cK$ is given in Definition \ref{generatorsforA(SO(3),c)} and $d\cA(SO(3),\cT)$ is considered with the injective model structure.

\[
\xymatrix@C=2pc{
d\cA(SO(3),\cT) \ 
\ar@<+1ex>[r]^(.4){\widetilde{F}}
&
\widetilde{F}(\cK)\endash\mathrm{cell}\endash d\cA(O(2),\tcT)\ 
\ar@<+0.5ex>[l]^(.6){\widetilde{R}}
}
\]

\end{cor}

\begin{rmk}\label{rmk:cells_from_tildaFK}Let us calculate the cells from $\widetilde{F}(\cK)$ (ignoring suspensions as they work in the same way in both categories): 
$$\widetilde{F}(\sigma_1)=\widetilde{F}(\bQ_1 \lra 0)= \widetilde{\bQ}\oplus \Sigma^{-2}\bQ \lra 0$$ where $c$ sends $\widetilde{\bQ}$ to $\bQ$ (both copies of $\bQ$ are in the place corresponding to the trivial subgroup)  and 
$$\widetilde{F}(\sigma_{(H)})=\widetilde{F}(\bQ[W]_{(H)} \lra 0)=\bQ[W]_H \lra 0$$ where the left $\bQ[W]$  is in the place corresponding to $(H)$ and the resulting $\bQ[W]$ is in the place corresponding to $H$. This holds for all $(H) \in \oF$ except for $H=1$. For the torus
$$\widetilde{F}(\sigma_{(T)})=\widetilde{F}(M \lra \cE^{-1}\cO_{\cF} \otimes \bQ[W])= \Sigma^2 \widetilde{\bQ} + \cO_\cF \otimes \bQ[W] \lra  \cE^{-1}\cO_{\cF} \otimes \bQ[W]$$ where $c$ acts on $ \widetilde{\bQ}$ in degree $2$ ($\widetilde{\bQ}$ is in the place corresponding to the trivial subgroup) by sending it to $\bQ \subseteq \bQ[W]$ in degree $0$ and the map is the inclusion.
\end{rmk}

\subsection{Restriction to the toral part of rational  $O(2)$--spectra}\label{section_toral_reduction_to_O2} 

The idea for the comparison is to restrict the toral part of rational $SO(3)$--spectra to the toral part of rational $O(2)$--spectra using the functor $i^\ast$ as a left adjoint. Recall that the adjunction ($SO(3)_+\wedge _{O(2)}-, i^\ast$) is not a Quillen pair for the  model categories localised at the idempotents corresponding to the toral parts, see Proposition \ref{notQErest_ind}. 

We use the proof from \cite{BarnesO(2)new} giving an algebraic model for the toral part of rational $O(2)$--spectra, cellularising every step of the zig--zag of Quillen equivalences presented there. This way we obtain an algebraic model for the toral part of rational $O(2)$--spectra cellularised at the derived images of generators for $L_{e_{\cT}S_\bQ}(SO(3)\endash\mathrm{Sp})$. This gives an algebraic model, however not very explicit. We finish this section by simplifying this category in Theorem \ref{koncowePorownanieModeli} and showing that it is Quillen equivalent to $d\cA(SO(3), \cT)$ with the injective model structure.

We start by establishing generators for the toral part of rational $SO(3)$--spectra. We used the notation $\cK$ in Definition \ref{generatorsforA(SO(3),c)} for the generators on the algebraic side. We will use the notation $K$ for the generators on the topological side. We will end this section showing that the derived images of the topological generators $\mathrm{im}(K)$ are precisely the algebraic generators $\cK$ in $d\cA(SO(3), \cT)$.

\begin{prop}\label{homComCyclicGen}A set $K$ consisting of all suspensions and desuspensions of one $SO(3)$--spectrum $$\sigma_{n}=SO(3)_+\wedge_{C_n}e_{C_n}S^0$$ for every natural $n>0$ and an $SO(3)$--spectrum $SO(3)/SO(2)_+$ is a set of cofibrant homotopically compact generators for the category $L_{e_{\cT}S_\bQ}(SO(3)\endash\mathrm{Sp})$. 
\end{prop}
\begin{proof} 
First consider a set $L$ consisting of all suspensions and desuspensions of one $SO(3)$--spectrum $SO(3)/C_{n+}$ for every natural $n>0$ and an $SO(3)$--spectrum $SO(3)/SO(2)_+$. 
All objects in $L$ are homotopically compact in $L_{e_{\cT}S_\bQ}(SO(3)\endash\mathrm{Sp})$ since they are in $SO(3)\endash\mathrm{Sp}$ and fibrant replacement in $L_{e_{\cT}S_\bQ}(SO(3)\endash\mathrm{Sp})$ commutes with coproducts.
This is a set of generators for $L_{e_{\cT}S_\bQ}(SO(3)\endash\mathrm{Sp})$ by \cite[Chapter IV, Proposition 6.7]{MandellMay}. 
Since  $$SO(3)/{C_n}_+=\bigvee_{C_m \subseteq C_n}\sigma_m $$ which is a consequence of \cite[Lemma 2.1.5]{GreenleesS1}, the set $K$ is a set of homotopically compact generators for $L_{e_{\cT}S_\bQ}(SO(3)\endash\mathrm{Sp})$.
\end{proof}

Next we restrict to the toral part of rational $O(2)$--spectra.

\begin{thm}\label{CyclicRestrictionEquivalence}
The following adjunction
\[
\xymatrix{
i^\ast \ :\ L_{e_{\cT}S_\bQ}(SO(3)\endash\mathrm{Sp})\  \ar@<+1ex>[r] & \ i^\ast(K)\endash\mathrm{cell}\endash L_{e_{\tcT}S_\bQ}(O(2)\endash\mathrm{Sp}) \ : F_{O(2)}(SO(3)_+,-) \ar@<+0.5ex>[l]
}
\]
is a Quillen equivalence, where the idempotent on the right hand side corresponds to the family of all cyclic subgroups of $O(2)$.
\end{thm}

\begin{proof} The fact that this is a Quillen adjunction follows from Proposition \ref{cyclicAdj} and the cellularisation principle \cite[Theorem 2.1]{GreenShipCell} for $K$ and $i^\ast(K)$. Notice that since $K$ was a set of generators for the category $L_{e_{\cT}S_\bQ}(SO(3)\endash\mathrm{Sp})$, the cellularisation with respect to $K$ will not change this model structure.

All cells from $K$ are homotopically compact and cofibrant by Proposition \ref{homComCyclicGen}. We need to check that their images with respect to $i^\ast$ are homotopically compact in  $L_{e_{\tcT}S_\bQ}(O(2)\endash\mathrm{Sp})$, i.e. suspension spectra of $SO(3)/C_{n+}$ for all $n$ and $SO(3)/SO(2)_+$ as toral $O(2)$--spectra. It is enough to show that they are homotopically compact as $O(2)$--spectra, which follows from the fact that a smooth, compact $G$-manifold admits a structure of a finite $G$--CW complex (\cite{Illman}) and a suspension spectrum of a finite $G$-CW complex is homotopically compact.  It thus follows that the images of the summands $\sigma_n$ are also homotopically compact and cofibrant in $L_{e_{\tcT}S_\bQ}(O(2)\endash\mathrm{Sp})$.

It remains to show that the components of the derived unit maps at generators are weak equivalences. For this, it is enough to check the induced map on the level of homotopy categories. This is equivalent to showing that the derived functor $Li^\ast$ is an isomorphism on hom--sets. This holds by \cite[Theorem 6.1]{GreenleesSO3} which states that if $X\cong e_{\cT}X$ then $Li^\ast$ is an isomorphism:
$$[X,Y]^{SO(3)}\lra e_{\tcT}[i^*X,i^*Y]^{O(2)}$$
which implies that $$Li^*: [X,Y]^{L_{e_{\cT}}SO(3)}\cong[e_{\cT}X,e_{\cT}Y]^{SO(3)} \lra e_{\tcT}[i^*(e_{\cT}X),i^*(e_{\cT}Y)]^{O(2)}\cong [i^*X,i^*Y]^{L_{e_{\tcT}}O(2)}$$ is an isomorphism, where the superscript ${L_{e_{\cT}}SO(3)}$ was used to mean the homotopy category of $L_{e_{\cT}S_\bQ}(SO(3)\endash\mathrm{Sp})$. Similarly, the superscript $L_{e_{\tcT}}O(2)$ was used to mean the homotopy category of $L_{e_{\tcT}S_\bQ}(O(2)\endash\mathrm{Sp})$.
Thus the adjunction is a Quillen equivalence.
 
\end{proof}

\begin{rmk}\label{rmk:generalGtoN}
The result above generalises to any compact Lie group $G$. The restriction--coinduction adjunction is a Quillen equivalence between the toral part of rational $G$--spectra and a certain cellularisation of the toral part of rational $N$--spectra, where $N$ is the normaliser of the maximal torus in $G$. This is used in \cite{BGK} to provide an algebraic model for the toral part of rational $G$--spectra for any compact Lie group $G$.
\end{rmk}

Since the Quillen equivalence above provides a link between the toral part of rational  $SO(3)$--spectra and the toral part of rational  $O(2)$--spectra we use the result of \cite{BarnesO(2)new}.
\begin{thm} \cite[Corollary 4.22]{BarnesO(2)new}
There is a zig-zag of Quillen equivalences between $L_{e_{\tcT}S_\bQ}(O(2)\endash\mathrm{Sp})$ and $d\cA(O(2),\tcT)$, where $d\cA(O(2),\tcT)$ is considered with the dualizable model structure.
\end{thm}

To provide an algebraic model for rational  $SO(3)$--spectra we need to cellularise every step of the zig-zag from  \cite[Section 4]{BarnesO(2)new} with respect to derived images of $i^*(K)$ from  Theorem \ref{CyclicRestrictionEquivalence}. Cellularisation preserves Quillen equivalences and gives the following result.

\begin{thm}\label{summary_cyclic_SO3}There is a zig-zag of Quillen equivalences between $L_{e_{\cT}S_\bQ}(SO(3)\endash\mathrm{Sp})$ and $\mathrm{im}(K)\endash\mathrm{cell}\endash d\cA(O(2),\tcT)$, where $d\cA(O(2),\tcT)$ is considered with the dualizable model structure. Here $\mathrm{im}(K)$ denotes the derived image under the zig-zag of Quillen equivalences described in \cite[Section 4]{BarnesO(2)new} of the set of cells $K$ described in Proposition \ref{homComCyclicGen}. 
\end{thm}

Theorem \ref{summary_cyclic_SO3} already gives an algebraic model for the toral part of rational $SO(3)$--spectra. However, it is not easy to work with a cellularisation of a model category. Thus we show that the model  above is Quillen equivalent to the simpler, algebraic category $d\cA(SO(3),\cT)$ described in Section \ref{The category A(SO(3),c)}. To do this, we first switch to the cellularisation of the injective model structure.

\begin{lem}The identity adjunction between $\mathrm{im}(K)\endash\mathrm{cell}\endash d\cA(O(2),\tcT)$ where $d\cA(O(2),\tcT)$ is equipped with the dualizable model structure and $\mathrm{im}(K)\endash\mathrm{cell}\endash d\cA(O(2),\tcT)$ where $d\cA(O(2),\tcT)$ is equipped with the injective model structure is a Quillen equivalence.
\end{lem}
\begin{proof}The result follows from the fact that the identity adjunction was a Quillen equivalence between $d\cA(O(2),\tcT)$ with the dualizable model structure and $d\cA(O(2),\tcT)$ with the injective model structure.
\end{proof}

\begin{lem}\label{obliczenianacoidageneratory}The set $\mathrm{im}(K)$ consists of the same objects as $\widetilde{F}(\cK)$, where $\cK$ is the set described in Definition \ref{generatorsforA(SO(3),c)} and $\mathrm{im}(K)$ denotes the derived image under the zig-zag of Quillen equivalences described in \cite[Section 4]{BarnesO(2)new} of the set of cells $K$ described in Proposition \ref{homComCyclicGen}.
\end{lem}
\begin{proof}
First we show that for every $n>1$ $\sigma_n$ is weakly equivalent in $L_{e_{\tcT}S_\bQ}(O(2)\endash\mathrm{Sp})$ to $O(2)\wedge_{C_n}e_{C_n}S^0$. The map is induced by the inclusion of $O(2)$ into $SO(3)$ and we will show that it induces an isomorphism on all $\pi_*^H$ for $H\in \tcT$. We will use the notation $N=O(2)$ and $G=SO(3)$ below. We have
$$\pi_*^H(N\wedge_{C_n}e_{C_n}S^0)=[N/H_+,F_{C_n}(N_+, S^{L_N(C_n)}\wedge e_{C_n}S^0)]^N=[N/H_+, S^{L_N(C_n)}\wedge e_{C_n}S^0]^{C_n}.$$
Here ${L_N(C_n)}$ is the tangent $C_n$--representation at  the identity  coset of $N/{C_n}$ and thus is the $1$--dimensional trivial representation. Since the codomain has only geometric fixed points for $H=C_n$ we get a non zero result only for $H= C_n$:
$$[\Phi^{C_n}(N/{C_n}_+),\Phi^{C_n}(S^{L_N(C_n)})]=[S^1\vee S^1,S^1]=\Sigma(\bQ[W]).$$
Here ${L_N(C_n)}$ is the tangent $C_n$--representation at  the identity  coset of $N/{C_n}$ and thus is the $1$--dimensional trivial representation. Similarly we have
$$\pi_*^H(G\wedge_{C_n}e_{C_n}S^0)=[G/H_+,F_{C_n}(G_+, S^{L_G(C_n)}\wedge e_{C_n}S^0)]^G=[G/H_+, S^{L_G(C_n)}\wedge e_{C_n}S^0]^{C_n}$$
and since the codomain has only geometric fixed points for $H=C_n$ we get a non--zero result only for $H= C_n$:
$$[\Phi^{C_n}(G/{C_n}_+),\Phi^{C_n}(S^{L_G(C_n)})]=[S^1\vee S^1,S^1]=\Sigma(\bQ[W]).$$
Notice that ${L_G(C_n)}$ is $3$--dimensional, but it has a one dimensional $C_n$--fixed subspace. 

The images of the cells in $\cA(O(2),\tcT)$ are therefore 
$$\mathrm{im}( G\wedge_{C_n}e_{C_n}S^0)= \mathrm{im}(N\wedge_{C_n}e_{C_n}S^0)= (\Sigma \bQ[W]_{C_n} \lra 0)$$ by \cite[Example 5.8.1]{GreenleesS1} where $\Sigma \bQ[W]$ is in the place $C_n$. 

Now we will use the functors $\pi_*^\cA$ described in \cite[Theorem 5.6.1 and Lemma 5.6.2]{GreenleesS1}.
Since $SO(3)_+$ is free we get 
\begin{multline}\pi_*^\cA(SO(3)_+)=(\pi_*^T(SO(3)_+) \lra 0) =( \pi_*(\Sigma SO(3)/T_+) \lra 0)=\\
( \pi_*(\Sigma S(\bR^3)_+)\lra 0)=(\Sigma^3\widetilde{\bQ} \oplus \Sigma \bQ \lra 0)
\end{multline}
where $\Sigma^3\widetilde{\bQ} \oplus \Sigma \bQ$ is in the place corresponding to the trivial subgroup $1$ and $c$ sends $\tilde{\bQ}$ in degree 3 to $\bQ$ in degree 1.

Finally $SO(3)/T_+=S(\bR^3)_+$  is built as an $O(2)$--space from the following cells.
$$N/T_+ \vee N/D_{2+} \cup N_+\wedge e^1$$
Thus the cofibre sequence
$$N_+ \lra N/T_+\vee N/D_{2+} \lra G/T_+$$
gives the long exact sequence
$$...\lra (\Sigma \bQ[W] \lra 0) \lra (\cO_{\cF}[W]\lra \cE^{-1}\cO_{\cF}\otimes \bQ[W]) \oplus (\Sigma \bQ \lra 0)\lra \mathrm{im}(G/T_+)\lra ...$$
and hence $$\mathrm{im}(G/T_+)= \Sigma^2 \widetilde{\bQ} + \cO_\cF \otimes \bQ[W] \lra  \cE^{-1}\cO_{\cF} \otimes \bQ[W]$$ where $c$ acts on $ \widetilde{\bQ}$ in degree $2$ ($\widetilde{\bQ}$ is in the place corresponding to the trivial subgroup) by sending it to $\bQ \subseteq \bQ[W]$ in degree $0$ and the map is the inclusion.

Notice that these images are exactly the cells (up to suspension) in $\widetilde{F}(\cK)$ (see Remark \ref{rmk:cells_from_tildaFK}), which finishes the proof.
\end{proof}

\begin{thm}\label{koncowePorownanieModeli}The adjunction 
 \[
\xymatrix{
\widetilde{F} \ :\  \mathrm{d}\cA(SO(3),\cT)  \  \ar@<+1ex>[r] & \ \mathrm{im}(K)\endash\mathrm{cell}\endash d\cA(O(2),\tcT) \ :\ \widetilde{R} \ar@<+0.5ex>[l]
} 
\]
defined after the Theorem \ref{adjunction1} is a Quillen equivalence, where both categories (before cellularisation on the right) are equipped with the injective model structure. Here $\mathrm{im}(K)$ denotes the derived image under the zig-zag of Quillen equivalences described in \cite[Section 4]{BarnesO(2)new} of the set of cells $K$ described in Proposition \ref{homComCyclicGen}. 
\end{thm}

\begin{proof} It is enough to show that $\mathrm{im}(K)$ consists of the same objects as $\widetilde{F}(\cK)$, where $\cK$ is the set described in Definition \ref{generatorsforA(SO(3),c)}, which we established in Lemma \ref{obliczenianacoidageneratory}. The result follows then from Corollary \ref{koncoweCor}.

\end{proof}

We summarize the results of this section.

\begin{thm}\label{thm:toral_result}
There is a zig--zag of Quillen equivalences between $L_{e_{\cT}S_\bQ}(SO(3)\endash\mathrm{Sp})$ and $d\cA(SO(3),\cT)$.
\end{thm}

\section{The dihedral part}\label{Dihedral part_chapter}
The algebraic model for the dihedral part of rational $SO(3)$--spectra is almost identical to the algebraic model of the dihedral part of rational $O(2)$--spectra presented in \cite[Section 5]{BarnesO(2)new}. The difference comes from two things. First, in $SO(3)$ every dihedral subgroup of order 2, namely $D_2$ is conjugate to cyclic subgroups $C_2$ and thus is already taken into account in the toral part. Second, the normaliser of $D_4$ in $SO(3)$ is a subgroup $\Sigma_4$. For those reasons we exclude subgroups conjugate to $D_2$ and subgroups conjugate to $D_4$ from the dihedral part $\mathcal{D}$. Excluding $D_2$ and $D_4$ from the dihedral part $\cD$ allows us to deduce that the information captured by subgroups of $SO(3)$ that are in $\cD$ is the same as captured by subgroups of $O(2)$ that are in $\tcD \setminus \{D_2, D_4\}$, see Proposition \ref{generatorsDihedralI}. This leads to the reduction of dihedral part of rational $SO(3)$--spectra to the (part of the) dihedral part of rational $O(2)$--spectra in Theorem \ref{iQuillenEqDih}.

We know from \cite{GreenleesSO3} that the model for the homotopy category of the dihedral part of rational $SO(3)$--spectra is of the form of certain sheaves over an orbit space for $\mathcal{D}$, denoted further by $\cA(SO(3), \cD)$. Section \ref{A(D)} discusses this category as well as the category of chain complexes in $\cA(SO(3), \cD)$; $\Ch(\cA(SO(3), \cD))$. In Section \ref{subsec:comparison_dihedral} we present the comparison between the dihedral part of rational $SO(3)$--spectra and its algebraic model $\Ch(\cA(SO(3), \cD))$.

\subsection{Categories  $\cA(SO(3), \cD)$ and $\Ch(\cA(SO(3), \cD))$}\label{A(D)}
First we recall the construction of $\cA(SO(3), \cD)$ (see also \cite{GreenleesSO3}), then we present the model structure on $\Ch(\cA(SO(3), \cD))$ and recall a set of homotopically compact generators for this model category.  

Material in this section is based on \cite[Section 5.1]{BarnesO(2)new}. There is a slight difference between the definition of $\cA(O(2), \tcD)$ presented there ($\cA(O(2), \cD)$  is the notation used in \cite{BarnesO(2)new} for this category) and $\cA(SO(3), \cD)$ below, namely we start indexing modules from $k=3$, which corresponds to $D_6=D_{2k}$. Indexing in \cite{BarnesO(2)new} starts from $1$. 
 
Let $W$ be the group of order two. 
\begin{defn} Define a category $\cA(SO(3), \cD)$ as follows.
 
An object $M$ consists of a $\bQ$--module $M_\infty$, a collection $M_k \in \bQ[W]$--mod for $k>2$ and a map (called the germ map) of $\bQ[W]$--modules $\sigma_M: M_\infty \lra \colim_{n>2} \prod_{k\geqslant n}M_k$, where the $W$--action on $M_\infty$ is trivial.

A map $f:M \lra N$ in $\cA(SO(3), \cD)$ consists of a map $f_\infty: M_\infty \lra N_\infty$ of $\bQ$--modules and a collection of maps of $\bQ[W]$--modules $f_k: M_k \lra N_k$ which commute with germ maps $\sigma_M$ and $\sigma_N$
\[
\xymatrix@C=2pc@R=2pc{M_\infty \ar[d]^{f_\infty} \ar@<+0.1ex>[r]^{\sigma_M\ \ \ \ \ } & \colim_{n>2}\prod_{k\geqslant n}M_k \ar[d]^{\colim_{n>2}\prod_{k\geqslant n}f_k}
\\ N_\infty \ar@<+0.1ex>[r]^{\sigma_N\ \ \ \ \ } & \colim_{n>2}\prod_{k\geqslant n}N_k \ .} 
\]

\end{defn}

\begin{defn} Define a category $\Ch(\cA(SO(3), \cD))$ to be the category of chain complexes in $\cA(SO(3), \cD)$ and $\mathrm{g}\cA(SO(3), \cD)$ to be the category of graded objects in $\cA(SO(3), \cD)$.
\end{defn}

An object $M$ of $\Ch(\cA(SO(3), \cD))$ consists of rational chain complex $M_\infty$, a collection of chain complexes of $\bQ[W]$--modules $M_k$ for $k>2$  and a germ map of chain complexes of $\bQ[W]$--modules $\sigma_M:  M_\infty \lra \colim_{n>2} \prod_{k\geqslant n}M_k$, where $W$--action on $M_\infty$ is trivial.

Note, that we used a chain complex notation here, unlike for the toral part, where we used $d\cA(SO(3),\cT)$ to mean differential objects in $\cA(SO(3),\cT)$. The difference between these two is that $\cA(SO(3), \cD)$ is not a graded category, and we introduce a grading taking chain complexes in $\cA(SO(3), \cD)$. On the other hand, $\cA(SO(3), \cT)$ is already graded, and we are interested in differential objects in $\cA(SO(3), \cT)$.

\begin{rmk} Since the only difference between our definition of $\cA(SO(3), \cD)$ and the one for $\cA(O(2), \tcD)$ lies in index $k$, all constructions for $\cA(SO(3), \cD)$ are analogous to the ones for $\cA(O(2), \tcD)$ presented in \cite{BarnesO(2)new}. 
\end{rmk}

It is helpful to consider several adjoint pairs involving the category $\Ch(\cA(SO(3), \cD))$. They are used to get a model structure on $\Ch(\cA(SO(3), \cD))$.

\begin{defn}\cite[Definition 5.9]{BarnesO(2)new}\label{adjunctionsA(D)}
Let $A \in \Ch(\bQ)$, $X\in \Ch(\bQ[W])$ and $M \in \Ch(\cA(SO(3), \cD))$. For a natural number $k>2$ we define the following functors:
\begin{itemize}
\item $i_k: \Ch(\bQ[W]) \lra \Ch(\cA(SO(3), \cD))$ by $(i_k(X))_\infty=0$ and $(i_k(X))_n=0$ for $n\neq k$ and $(i_k(X))_k=X$.
\item $p_k: \Ch(\cA(SO(3), \cD)) \lra \Ch(\bQ[W])$ by $p_k(M)=M_k$
\item $c: \Ch(\bQ) \lra \Ch(\cA(SO(3), \cD))$ by $(cA)_k=A$,  $(cA)_\infty = A$ and $\sigma_{cA}$ is the diagonal map into the product.
\end{itemize} 
Then $(i_k, p_k)$, $(p_k,i_k)$ and $(c,  \nplus^W)$ are adjoint pairs, where the functor $\nplus^W$ is given in \cite[Definition 5.6]{BarnesO(2)new}.
\end{defn}

The category $\Ch(\cA(SO(3), \cD))$ is bicomplete by \cite[Lemma 5.7]{BarnesO(2)new} so we can proceed to defining a model structure on it.

\begin{prop}\cite[Proposition 5.10]{BarnesO(2)new} There exists a model structure on the category $\Ch(\cA(SO(3), \cD))$ where $f$ is a weak equivalence or fibration if $f_\infty$ and each of the $f_k$ are weak equivalences or fibrations respectively. This model structure is cofibrantly generated and proper. 
\end{prop}
We call this model structure the {\bf projective model structure} on $\Ch(\cA(SO(3), \cD))$. 
The generating cofibrations are of the form $cI_\bQ$ and $i_kI_{\bQ[W]}$ for $k\geqslant3$ and generating acyclic cofibrations are of the form $cJ_\bQ$ and $i_kJ_{\bQ[W]}$ for $k\geqslant3$. Here $I_\bQ$ and $J_\bQ$ denote generating cofibrations and generating trivial cofibrations (respectively) for the projective model structure on $\Ch(\bQ)$, and $I_{\bQ[W]}$, $J_{\bQ[W]}$ denote  generating cofibrations and generating trivial cofibrations (respectively) for the projective model structure on $\Ch(\bQ[W])$.

We finish this section by giving a set of homotopically compact generators (recall Definitions \ref{compactobj} and \ref{def:hocompact}) for $\cA(SO(3), \cD)$.

\begin{lem}\cite[Lemma 5.11]{BarnesO(2)new}\label{genDihedralAlg} The set of objects $\cG_a$ consisting of  $i_k\bQ[W]$ for $k\geqslant 3$ and $c\bQ$ is a set of homotopically compact, cofibrant and fibrant generators for the category $\Ch(\cA(SO(3), \cD))$ with the projective model structure.
\end{lem}

\subsection{Comparison}\label{subsec:comparison_dihedral}
We begin by giving a set of homotopically compact, cofibrant generators for $L_{e_{\cD}S_\bQ}(SO(3)\endash\mathrm{Sp})$. We stick to the convention of writing $e_H$ for $e_{(H)_{SO(3)}}$.

\begin{lem}\label{lem:generatorsdihedralSO(3)}The set  $$\hat{\mathcal{G}}:=\{SO(3)/O(2)_+\} \cup \{e_{D_{2n}}{SO(3)/D_{2n}}_+ | n>2 \}$$ 
is a set of homotopically compact, cofibrant generators for $L_{e_{\cD}S_\bQ}(SO(3)\endash\mathrm{Sp})$.
\end{lem}
\begin{proof} The proof is the same as the proof of \cite[Lemma 5.14]{BarnesO(2)new}.
\end{proof}

To finish the discussion about generators, we show that the restriction functor $$i^*:L_{e_{\cD}S_\bQ}(SO(3)\endash\mathrm{Sp}) \lra L_{i^*(e_{\cD})S_\bQ}(O(2)\endash\mathrm{Sp})$$ preserves generators up to weak equivalence.
\begin{prop}\label{generatorsDihedralI}Recall that $i^*(e_\cD)$ is the idempotent in $\A(O(2))_\bQ$ corresponding to the characteristic function on subgroups $D_{2n}$ for $n>2$ and $O(2)$.
\begin{enumerate}
\item The map $f: O(2)/O(2)_+ \lra i^*(SO(3)/O(2)_+)$ induced by inclusion $O(2) \lra SO(3)$ is a weak equivalence in $L_{i^*(e_{\cD})S_\bQ}(O(2)\endash\mathrm{Sp})$.
\item The map $f_{2n}: e_{D_{2n}}O(2)/{D_{2n}}_+ \lra i^*(e_{D_{2n}}SO(3)/{D_{2n}}_+)$ for $n>2$ induced by inclusion $O(2) \lra SO(3)$ is a weak equivalence in $L_{i^*(e_{\cD})S_\bQ}(O(2)\endash\mathrm{Sp})$
\end{enumerate}
\end{prop}
\begin{proof} To show that the map $f: O(2)/O(2)_+ \lra i^*(SO(3)/O(2)_+)$ is a weak equivalence in the given model structure, we need to show that $i^*(e_{\cD})f$ is an equivariant rational $\pi_*$--isomorphism. Thus we need to check that for all subgroups $H \leq O(2)$ the $H$--geometric fixed points $$\Phi^H(i^*(e_{\cD}) f): \Phi^H(i^*(e_{\cD})O(2)/O(2)_+) \lra \Phi^H( i^*(e_{\cD})i^*(SO(3)/O(2)_+))$$ is a non-equivariant rational $\pi_*$-isomorphism.  

Since geometric fixed points commute with smash product and suspensions, for every subgroup $H \not \in (\tcD \setminus \{D_2,D_4\})$, $\Phi^H(i^*(e_{\cD}) f)$ is a trivial map between trivial objects. For $H=O(2)$ the map is an identity on $S^0$ since $O(2)$ is its own normaliser in $SO(3)$. For other $H \in (\tcD \setminus \{D_2,D_4\})$ it is an identity on $S^0$ since there is just one conjugacy class for every $n$ of $D_{2n}$ subgroups in $O(2)$ (and if $g\in SO(3)$ and $g\not \in O(2)$ then $g^{-1}D_{2n}g \not \subset O(2)$). 

Part 2 follows the same pattern, however the domain and codomain of the map $f_{2n}$ are already local in $L_{i^*(e_{\cD})S_\bQ}(O(2)\endash\mathrm{Sp})$, so $f \cong i^*(e_{\cD})f$. Since the idempotent used is $e_{D_{2n}}$ the only non-trivial geometric fixed points will be for the subgroup $H=D_{2n}$. The result follows from the fact that $N_{O(2)}D_{2n}=N_{SO(3)}D_{2n}$, which implies that the map on geometric fixed points for $D_{2n}$ is the identity on ${D_{4n}/D_{2n}}_+$, and that finishes the proof.
\end{proof}

To give an algebraic model for the dihedral part of rational $SO(3)$--spectra we firstly use the restriction--coinduction adjunction in the next theorem to move to a certain part of rational $O(2)$--spectra. Then we show that this part of rational $O(2)$--spectra is a localisation of the dihedral part of rational $O(2)$--spectra from \cite{BarnesO(2)new}. As a result the method of obtaining an algebraic model for this part presented in \cite{BarnesO(2)new} applies in our case almost verbatim.

\begin{thm}\label{iQuillenEqDih}Let $i: O(2) \lra SO(3)$ be an inclusion. Then the adjunction
\[
\xymatrix{
i^\ast \ :\ L_{e_{\cD}S_\bQ}(SO(3)\endash\mathrm{Sp})\  \ar@<+1ex>[r] & \ L_{i^*(e_{\cD})S_\bQ}(O(2)\endash\mathrm{Sp}) \ :\ F_{O(2)}(SO(3)_+,-) \ar@<+0.5ex>[l]
}
\]
is a Quillen equivalence. Note that the idempotent on the right hand side corresponds to the set of all dihedral subgroups of order greater than $4$ together with $O(2)$. 
\end{thm}
\begin{proof} This is a Quillen adjunction by Corollary \ref{localisedQAdjunctions} and moreover $i^*$ is a right Quillen functor by Proposition \ref{dihRightAdj}. 

We will use \cite[Corollary 1.3.16 part c]{Hovey}. To show that this adjunction is a Quillen equivalence first notice that $F_{O(2)}(SO(3)_+,-)$ preserves and reflects weak equivalences between fibrant objects. For any fibrant $X \in L_{i^*(e_{\cD})S_\bQ}(O(2)\endash\mathrm{Sp})$ and $H\in \tcD \setminus \{D_2,D_4\}$ we have natural isomorphisms $$[SO(3)/H_+,F_{O(2)}(SO(3)_+,X)]\cong [i^*SO(3)/H_+, X]\cong [O(2)/H_+,X]$$
where the second one follows from Proposition \ref{generatorsDihedralI}. Since weak equivalences between fibrant objects are detected by $H$--homotopy groups, $F_{O(2)}(SO(3)_+,-)$ preserves and reflects weak equivalences between fibrant objects.

We need to show that the derived unit 
$$Y \lra F_{O(2)}(SO(3)_+,\hat{f}i^\ast(Y))$$
is a weak equivalence on cofibrant objects in  $L_{e_{\cD}S_\bQ}(SO(3)\endash\mathrm{Sp})$. 
It is enough to check that the induced map 
$$[X,Y]^{L_{e_{\cD}S_\bQ}(SO(3)\endash\mathrm{Sp})}\cong [X, e_{\cD} Y]^{SO(3)} \lra [X,F_{O(2)}(SO(3)_+,\hat{f}i^\ast (e_{\cD} Y))]^{SO(3)} $$
is an isomorphism for every generator $X$ of $L_{e_{\cD}S_\bQ}(SO(3)\endash\mathrm{Sp})$ (see Lemma \ref{lem:generatorsdihedralSO(3)} for the set of generators). This map fits into the commuting diagram below.

\[
\xymatrix@R=2pc@C=5pc{
 [X, e_{\cD} Y]^{SO(3)} \ar[dr]^{i^\ast} \ar[d]& \  \\
  {[X, F_{O(2)}(SO(3)_+,\hat{f}i^\ast (e_{\cD} Y))]^{SO(3)}} \ar[r]^{\cong} & {[i^\ast X, \hat{f}i^\ast (e_{\cD}Y)]^{O(2)}} \\
}
\]

Since the horizontal map is an isomorphism it is enough to show that $i^\ast$ is an isomorphism on hom sets, where the domain is a generator for $L_{e_{\cD}S_\bQ}(SO(3)\endash\mathrm{Sp})$. We do this by using the second Quillen adjunction between these two categories, namely $(SO(3)_+\wedge_{O(2)}- ,i^*)$. 

Let $\eta$ denotes the categorical unit of the adjunction $(SO(3)_+\wedge_{O(2)}- ,i^*)$. The map $\eta$ on cofibrant generators is of the form  $$\eta_{e_HO(2)/H_+}: e_HO(2)/H_+ \lra e_Hi^*(SO(3)/H_+)$$ 
induced by an inclusion $O(2) \lra SO(3)$.  By Proposition \ref{generatorsDihedralI} this is a weak equivalence in $L_{i^*(e_{\cD})S_\bQ}(O(2)\endash\mathrm{Sp})$ for all $H$ in $\cD$ and thus $-\circ \eta$ induces an isomorphism in homotopy category.
We have the following commuting diagram

\[
\xymatrix@R=2pc@C=5pc{
  [e_HSO(3)/H_+, e_{\cD}Y]^{SO(3)} \ar[dr]^{i^\ast} \ar[d]_{\cong}& \  \\
  {[e_HO(2)/H_+, i^*(e_{\cD}Y)]^{O(2)}}& {[i^*(e_HSO(3)/H_+), i^*(e_{\cD}Y)]^{O(2)}}  \ar[l]^{-\circ \eta}\\ }
\]
where $H$ above denotes a finite dihedral subgroup or $O(2)$ (When $H$ is $O(2)$ we understand $e_H$ as $e_{\cD}$). 

 It follows that $i^*$ is an isomorphism on hom sets and thus the derived unit of the adjunction where $i^*$ is the left adjoint is a weak equivalence in $L_{e_{\cD}S_\bQ}(SO(3)\endash\mathrm{Sp})$ which finishes the proof.
\end{proof}

To obtain the algebraic model for rational $SO(3)$--spectra it is enough to get one for $L_{i^*(e_{\cD})S_\bQ}(O(2)\endash\mathrm{Sp})$. We use the comparison method presented in \cite{BarnesO(2)new} for the dihedral part of rational $O(2)$--spectra in this case.

\begin{thm}\label{summaryDihedral}There exist a zig-zag of Quillen equivalences from $L_{i^*(e_{\cD})S_\bQ}(O(2)\endash\mathrm{Sp})$ to $\Ch(\cA(SO(3), \cD))$.
\end{thm}
\begin{proof} Notice that $L_{i^*(e_{\cD})S_\bQ}(O(2)\endash\mathrm{Sp})$ is a localisation of the dihedral part of rational $O(2)$--spectra $L_{e_{\tcD}S_\bQ}(O(2)\endash\mathrm{Sp})$ at an idempotent $i^*(e_{\cD})$, since $i^*(e_{\cD})e_{\tcD}=i^*(e_{\cD})$. 
The set $$\widetilde\cG:=\{O(2)/O(2)_+\} \cup \{e_{D_{2n}}{O(2)/D_{2n}}_+ | n>2 \}$$
is a set of homotopically compact, cofibrant generators for $L_{i^*(e_{\cD})S_\bQ}(O(2)\endash\mathrm{Sp})$ by the same argument as in \cite[Lemma 5.14]{BarnesO(2)new}. 

Thus it is enough to use the proof of \cite[Theorem 5.18]{BarnesO(2)new} based on use the tilting theorem of Schwede and Shipley, \cite[Theorem 5.1.1]{SchwedeShipleyMorita} 
 restricted to the set of generators $\widetilde{\cG}$ for $L_{i^*(e_{\cD})S_\bQ}(O(2)\endash\mathrm{Sp})$ on one hand and the set of generators  $\cG_a$ (see Lemma \ref{genDihedralAlg}) on the algebraic side. This shows that $L_{i^*(e_{\cD})S_\bQ}(O(2)\endash\mathrm{Sp})$ is Quillen equivalent to the category $\Ch(\cA(SO(3),\cD))$. 
\end{proof}

Theorem \ref{iQuillenEqDih} and Theorem \ref{summaryDihedral} give the algebraic model for the dihedral part of rational $SO(3)$--spectra.

\begin{thm}\label{thm:summary_d}There is a zig-zag of Quillen equivalences between $L_{e_{\cD}S_\bQ}(SO(3)\endash\mathrm{Sp})$ and $\cA(SO(3), \cD)$.
\end{thm}


\section{The exceptional part}\label{Exceptional part_chapter}
The last part of rational $SO(3)$--spectra, $L_{e_{\cE} S_{\bQ}}(SO(3)\endash\mathrm{Sp})$, captures the bahaviour of conjugacy classes of five subgroups: $SO(3)$, $\Sigma_4$, $A_4$, $A_5$ and $D_4$, see Section \ref{subgroups}.

\begin{defn}\label{defn:exc_subgp}\cite[Definition 2.1]{KedziorekExceptional} Recall, that a subgroup $H$ of $G$ is \textbf{exceptional} if three conditions are satisfied: 
\begin{itemize}
\item there is an idempotent $e_{(H)} \in \A(G)_\bQ$ corresponding to the conjugacy class of $H$
\item the Weyl group of $H$, $N_GH/H$ is finite and 
\item $H$ does not contain any subgroup $K$, such that $H/K$ is a (non-trivial) torus.
\end{itemize}
\end{defn}

Notice that all subgroups in this part satisfy the definition above, hence the name \textbf{exceptional part}. By \cite[Theorem 4.4]{Barnes_Splitting} we have the following result. 
\begin{prop}\label{prop:exceptional_splitting}
There is a strong symmetric monoidal Quillen equivalence 
\[
\xymatrix{
\triangle\ :\ L_{e_{\cE}S_\bQ}SO(3) \endash\mathrm{Sp}_\bQ\ \ar@<+1ex>[r] & \ \prod_{(H), H\in \cE}\ L_{e_{(H)_{SO(3)}}S_\bQ}(SO(3)\endash\mathrm{Sp})\ :\Pi \ar@<+0.5ex>[l]
}
\]
\end{prop}

First we recall some details on what will be the building block of the algebraic model for the exceptional part, i.e. the category $\Ch(\bQ[W_GH])$ of chain complexes of $\bQ[W_GH]$--modules and then we summarise the monoidal comparison from \cite{KedziorekExceptional}.

\subsection{The category $\Ch(\bQ[W])$}\label{chainExcep}
Suppose $W$ is a finite group. The category of chain complexes of left $\bQ[W]$--modules can be equipped with the projective model structure, where weak equivalences are homology isomorphisms and fibrations are levelwise surjections. 
This model structure is cofibrantly generated by \cite[Section 2.3]{Hovey}. 

Note that $\bQ[W]$ is not generally a commutative ring, however it is a Hopf algebra with cocommutative coproduct given by $\Delta : \bQ[W] \lra \bQ[W] \otimes \bQ[W]$, $g \mapsto g \otimes g$. This allows us to define an associative and commutative tensor product on $\Ch(\bQ[W])$, namely tensor over $\bQ$, where the $W$--action on the $X \otimes_\bQ Y$ is diagonal. The unit is a chain complex with $\bQ$ at the level 0 with trivial $W$--action and zeros everywhere else and it is cofibrant in the projective model structure. 
The monoidal product defined this way is closed, where the internal hom is given by a formula for an internal hom in $\bQ$--modules with $W$--action given by conjugation. 

By \cite[Proposition 4.3]{BarnesFiniteG} the category $\Ch(\bQ[W])$ is a monoidal model category satisfying the monoid axiom. 

\subsection{Monoidal comparison}

The following result is the main theorem of \cite{KedziorekExceptional}. 
\begin{thm}\label{thm:exceptional}Suppose $G$ is any compact Lie group. Then there is a zig-zag of symmetric monoidal Quillen equivalences from $L_{e_{(H)_{G}}S_\bQ}(G\endash\mathrm{Sp})$ of rational  $G$--spectra over an exceptional subgroup $H$  to ${\Ch(\bQ[W_{G}H])}$ equipped with a projective model structure.
\end{thm}

We apply the result above in case $G=SO(3)$ to get the algebraic model for the exceptional part of rational $SO(3)$--spectra.
\begin{thm}\label{thm:exceptional_part}
There is a zig-zag of symmetric monoidal Quillen equivalences from $L_{e_{\cE}S_{\bQ}}(SO(3)\endash\mathrm{Sp})$ to ${\prod_{(H),H\in \cE} \Ch(\bQ[W_{SO(3)}H])}$
\end{thm}
\begin{proof} This follows from Proposition \ref{prop:exceptional_splitting} and Theorem \ref{thm:exceptional}.
\end{proof}

Below we present a short sketch of steps in the monoidal comparison for rational $G$--spectra over an exceptional subgroup to outline general ideas. We refer the reader to \cite{KedziorekExceptional} for all the details. 

Fix an exceptional subgroup $H$ in $G$. First we move from the category ${L_{e_{(H)_G}S_{\bQ}}(G\endash\mathrm{Sp})}$ to the category $L_{e_{(H)_N} S_{\bQ}}(N\endash\mathrm{Sp})$ using the restriction--coinduction adjunction, where $N$ denotes the normalizer $N_GH$. The second step is to use the fixed point--inflation adjunction between $L_{e_{(H)_N} S_{\bQ}}(N\endash\mathrm{Sp})$  and $L_{e_1 S_{\bQ}}(W\endash\mathrm{Sp})$, where $W$ denotes the Weyl group $N/H$. Recall that $W$ is finite, as $H$ is an exceptional subgroup of $G$. Next we use the restriction of universe to pass from $L_{e_1 S_{\bQ}}(W\endash\mathrm{Sp})$ to the category $\mathrm{Sp}[W]$ of rational orthogonal spectra with $W$--action. We then pass to symmetric spectra with $W$--action using the forgetful functor from orthogonal spectra and then to $H\bQ$-modules with $W$--action in symmetric spectra. From here we use \cite[Theorem 1.1]{ShipleyHZ} to get to $\Ch(\bQ)[W]$, the category of rational chain complexes with $W$--action, which is equivalent as a monoidal model category to $\Ch(\bQ[W])$, the category of chain complexes of $\bQ[W]$-modules. That gives an algebraic model which is compatible with the monoidal product, i.e. this zig-zag of Quillen equivalences induces a strong monoidal equivalence on the level of homotopy categories.

\end{document}